\documentclass[11pt]{article}
\usepackage{graphicx}
\usepackage{amssymb}
\usepackage{amsmath}
\usepackage{amsthm}
\usepackage{amscd}
\usepackage{psfrag}
\usepackage{float} 
\usepackage{appendix}
\usepackage[all]{xy}
\usepackage{url}

\newtheorem{prop}{Proposition}[section] 
\newtheorem{algo}{Algorithm}[section] 

\newtheorem{remark}{Remark}[section]
\def\begrem{\begin{remark}\rm}
\def\endrem{\null\hfill\blackbox\end{remark}}

\def\be{\begin{equation}}
\def\ee{\end{equation}}

\newcommand{\dd}{\textnormal{d}}

\title{A particle micro-macro decomposition based numerical scheme for collisional kinetic equations in the diffusion scaling.}
\author{Ana\"is Crestetto\footnotemark[1] \and Nicolas Crouseilles\footnotemark[2] \and Mohammed Lemou\footnotemark[3]}

\begin{document}
\maketitle

\footnotetext[1]
{INRIA Rennes - Bretagne Atlantique, IPSO team \& Laboratoire de Math\'ematiques Jean Leray, CNRS UMR 6629, 
Universit\'e de Nantes, France. E-mail: \url{anais.crestetto@univ-nantes.fr}}

\footnotetext[2]
{INRIA Rennes - Bretagne Atlantique, IPSO team \& Institut de Recherche Math\'ematiques de Rennes, CNRS UMR 6625, Universit\'e de Rennes 1, France \& ENS Rennes. E-mail: \url{nicolas.crouseilles@inria.fr}}

\footnotetext[3]
{Institut de Recherche Math\'ematiques de Rennes, CNRS UMR 6625, Universit\'e de Rennes 1, France \& ENS Rennes \& INRIA Rennes - Bretagne Atlantique, IPSO team. E-mail: \url{mohammed.lemou@univ-rennes1.fr}}

\begin{abstract}
In this work, we derive particle schemes, based on micro-macro decomposition, for linear kinetic equations 
in the diffusion limit. Due to the particle approximation of the micro part, a splitting between the transport and the collision part  has to be performed, and the stiffness of both these two parts 
 prevent from uniform stability. To overcome this difficulty, the micro-macro system is reformulated into a continuous PDE 
 whose coefficients are no longer stiff, and  depend  on the time step $\Delta t$ in a consistent way. This non-stiff  reformulation of the micro-macro  system 
allows the use of  standard particle approximations for the transport part, and extends the work in \cite{ccl} where a particle approximation has  been applied using a micro-macro decomposition on kinetic equations in the fluid scaling. Beyond the so-called asymptotic-preserving property which is satisfied by our schemes, they significantly reduce the inherent noise of traditional particle methods, and 
they have a  computational cost which decreases as the system approaches the diffusion limit.  
\end{abstract}

%################################################################
%################################################################
%###################Beginning of the report######################
%################################################################
%################################################################

\section{Introduction}
Particle systems appearing in plasma physics or radiative transfer can be described at different scales. When the system is far from its thermodynamical equilibrium, a kinetic description is necessary. Particles are then represented by a distribution function $f$ which depends on time $t\geq 0$, position $x\in\mathbb{R}^d$ and velocity $v\in V\subset\mathbb{R}^d$,  $d\geq 1$. The distribution $f\left(t,x,v\right)$ satisfies a collisional kinetic equation. Particle methods are often used for simulating kinetic problems, especially in realistic $3$-dimensional situations, $d=3$. However, they are affected by numerical noise due to their probabilistic character. A simple way to reduce  this noise is  to increase the number of particles, but then the numerical cost increases as well. Other standard kinetic descriptions, as phase space grid methods, may require too much memory in the two or three dimensional framework. Otherwise, macroscopic descriptions depending only on $t$ and $x$ can be sufficient if the system stays near its thermodynamical equilibrium, and are  less expensive since their unknown does not depend on the velocity variable anymore. Beside the noisy character of standard particle methods, there is  an additional difficulty in kinetic descriptions which is linked to the presence of various  scales in the system.  Multi-scale phenomena may indeed  appear in  plasma devices or radiative transfer applications, depending on some physical parameters as for example the mean free path of particles or the Knudsen number denoted here by $\varepsilon$. This multi-scale character is often represented by stiff terms in the kinetic equation, and the general challenge  is to construct  efficient numerical methods for these multiscale kinetic equations: this means that, without numerically resolving the stiffness,  the numerical method must solve accurately the  kinetic regime, must have the right asymptotics in the high-stiffness limit (the so-called asymptotic preserving property) and  its computational cost  should decrease as the system approaches the equilibrium (a time diminishing property). Note that direct numerical methods whose parameters  resolve the smallest scale of size $\varepsilon$ are impossible to use, since they automatically involve an extremely high computational cost.

Several strategies have been proposed to overcome this strong constraint. Domain decomposition methods can be applied when we have different regions with different values of the scaling parameter, see \cite{dd,gjl, tallec, tiwari}. When the different scales are less clearly delimited, we have to develop kinetic schemes that naturally reduce to good approximations of the macroscopic problem when the system goes near its equilibrium, and overcome the stiffness. Such schemes are often called Asymptotic-Preserving (AP), see \cite{jin,bt,klar,dp,cl,larsen,lemou-note,mlb,lm,jpt,np,buet}. Mainly,  the numerical cost remains comparable to the one of the non-stiff kinetic problem, even when $\varepsilon\ll 1$.

Our goal is to design an efficient AP scheme, {\em using particles}, for the following kinetic radiative transport equation (RTE) 
in the diffusion scaling 
\begin{equation}
\partial_t f +\frac{1}{\varepsilon} v\partial_x f  = \frac{1}{\varepsilon^2}(\rho M - f),  \;\; f(t=0, x, v)=f_0(x, v), 
\label{eq:etrbgk}
\end{equation}
where $x\in \Omega\subset\mathbb{R}$, $\rho(t, x)=\frac{1}{2}\int_V f(t, x, v) \dd v$, $V=[-1, 1]$, $M=1$ 
and $f_0(x, v)$ is a given initial condition. Periodic boundary conditions are considered. 
It is well-known (see \cite{larsen, dgp}) that when $\varepsilon$ goes to zero, the distribution function 
$f(t, x, v)$ converges towards $\bar{\rho}(t, x) M(v)$, where $\bar{\rho}$ satisfies the following diffusion equation 
\begin{equation}
\partial_t \bar{\rho} -\frac{1}{3} \partial_{xx} \bar{\rho}  = 0, \;\; \bar{\rho}(t=0, x) = \frac{1}{2}\int_V f_0(x, v) \dd v. 
\label{eq:diff}
\end{equation}

An extension to the Vlasov-Poisson-BGK case is presented in Section \ref{ssec:e}. 
The kinetic equation is coupled to a Poisson equation for the electric field denoted by $E(t, x)$. More precisely, we consider
\begin{eqnarray}
\partial_t f +\frac{1}{\varepsilon} v\partial_x f  + \frac{1}{\varepsilon}E\partial_v f= \frac{1}{\varepsilon^2}(\rho M - f), \label{eq:vlasovbgk} \\
\partial_x E=\rho-1\label{eq:vlasovbgkpoisson},
\end{eqnarray}
where $x\in \Omega\subset\mathbb{R}$, $\rho(t, x)=\int_V f(t, x, v) \dd v$, $V=\mathbb{R}$, 
$M\left(v\right)=\frac{1}{\sqrt{2\pi}}\exp\left(-\frac{v^2}{2}\right)$ is the absolute Maxwellian and we consider periodic boundary conditions. 
Note that an additional condition $\int_\Omega E \dd x=0$ is imposed to obtain a well-posed problem.
When $\varepsilon$ goes to zero, the asymptotic model is a drift-diffusion equation satisfied by $\bar{\rho}(t, x)$ (see \cite{bat})
\begin{equation}
\partial_t \bar{\rho} - \partial_x(\partial_x \bar{\rho}-\bar{E}\bar{\rho})  = 0, 
\;\; \bar{\rho}(t=0, x) =\int_\mathbb{R} f_0(x, v) \dd v,  
\label{eq:drift-diff}
\end{equation}
where $\bar{E}$ is linked to $\bar{\rho}$ by the Poisson equation $\partial_x\bar{E}=\bar{\rho}-1$.

The strategy will be the use of the micro-macro decomposition (see \cite{liu, lm,bennoune,cl}). 
It consists in writing the distribution function as the sum of the equilibrium part and a rest. 
One can then derive a system of two equations: a kinetic one for the rest $g(t, x, v)$ 
and a macroscopic one 
for the equilibrium $\rho(t, x) M(v)$. AP micro-macro schemes for (\ref{eq:etrbgk}) 
have been proposed in \cite{lm,bennoune,cl}. 
These schemes consist in a semi-implicit phase space grid method for the kinetic part, coupled to a classical 
spatial grid method for the macro part. Our strategy in this work follows the strategy of \cite{ccl} in the case of a fluid scaling: 
we use particles to sample the kinetic part whereas an Eulerian solver is used to discretize the macro unknown. 
The main motivation of this strategy lies in the fact that the micro part $g$ converges to zero 
when $\varepsilon$ goes to zero, so that a very few number of particles can sample it. 
As a consequence, in this regime, the cost of the global micro-macro solver is almost the same as the cost of 
an asymptotic solver for \eqref{eq:diff}. 

In this work, we focus on a diffusion type scaling (as in \eqref{eq:etrbgk} or \eqref{eq:vlasovbgk}) so that 
an additional scale is involved compared to the fluid scaling considered in \cite{ccl}. In \cite{lm, cl}, a diffusion 
scaling was studied, but using a fully grid based solver. Hence, the stiffest term (of order $1/\varepsilon^2$) 
is considered implicit in time in the micro equation, which enables to stabilize the transport term (of order $1/\varepsilon$) 
and then to derive an AP scheme for \eqref{eq:etrbgk} and \eqref{eq:vlasovbgk}. 
The use of particles for the micro part prevents from a similar strategy since a splitting between the transport 
term (of order $1/\varepsilon$) and the source term needs to be done. Then, a uniform stable scheme is hard 
to obtain in this context. To overcome this difficulty, a suitable formulation of the original model \eqref{eq:etrbgk} 
is performed so that the stiff transport term $(1/\varepsilon) \, v\partial_x g$ becomes 
$\varepsilon/\Delta t (1-e^{-\Delta t/\varepsilon^2}) v\partial_x g$. This reformulation is correct 
up to $\Delta t^2$ (for fixed $\varepsilon>0$) and has the good behavior when $\varepsilon$ goes to zero 
(for fixed $\Delta t>0$). This formulation is the starting point of the design of micro-macro-particle based 
numerical schemes which enjoy the AP property and for which the numerical cost diminishes 
as $\varepsilon$ goes to zero. This approach is extended to the second-order (in time) 
and to the Vlasov-Poisson-BGK case \eqref{eq:vlasovbgk}-\eqref{eq:vlasovbgkpoisson}. 

%However, following \cite{ccl}, our goal here is to use 
%particles to discretize the micro part so that in the limit regime $\varepsilon\to 0$, 
%the numerical cost decreases. Indeed the micro part goes to 0 as $\varepsilon\to 0$ 
%so that very few particles are sufficient to sample it. The main difficulty 
%compared to phase space grid approaches \cite{lm, cl} remains in the fact that 
%the use of particles requires a splitting between transport and source terms whereas 
%in  \cite{lm, cl}, the stiffest (source) term is used to stabilize the stiff transport term. Another complication appears when considering Dirichlet boundary conditions in the RTE case, since it is not convenient to express Dirichlet conditions in terms of particles.

%Besides the schemes proposed in the literature, we will compare our particles-micro-macro scheme to a new micro-macro Eulerian one (detailed in Appendix \ref{subsec:eulerian}) and to an AP adaptation of the moment guided method (initially presented in \cite{ddp} and detailed in Appendix \ref{app:mgm}).

The sequel of the paper is organized as follows. In Section \ref{sec:1storder_00}, we recall the formal derivation 
of the asymptotic model of \eqref{eq:vlasovbgk} and \eqref{eq:etrbgk}. %which enables to introduce our notations. 
The first-order (in time) reformulation of \eqref{eq:etrbgk} is presented in Subsection \ref{sec:1storder} 
and its Lagrangian discretization in Subsection \ref{ssec:pic}. Its extension to a second-order in time model 
is detailed in Section \ref{sec:2ndorder_0}: the continuous model is presented in  Subsection \ref{sec:2ndorder} 
and its discretization is developed in Subsection \ref{subsec:ordre2_discr}. Section \ref{ssec:e} 
proposes an extension of our strategy to the Vlasov-Poisson-BGK system. Finally, Section \ref{sec:numres} is devoted to numerical simulations.
%Then, numerical results are presented in Section \ref{sec:numres} and a short conclusion 
%is given in Section \ref{sec:ccl}. Finally, Appendix \ref{subsec:eulerian} concerns the 
%Eulerian discretization of our first-order in time scheme for the ETR case and Appendix \ref{app:mgm} 
%presents the adaptation of the moment-guided method to our context, for comparison with our method.

\section{Diffusion asymptotics}
\setcounter{equation}{0}
\label{sec:1storder_00}
In this section, we recall the main steps of the derivation of the model obtained from \eqref{eq:vlasovbgk} 
when $\varepsilon$ goes to zero. To do so, we consider the micro-macro decomposition (see \cite{lm,bennoune,liu}) 
of $f$: $f(t, x, v)=\rho(t, x) M(v) + g(t, x, v)$, with $\rho(t, x)=\langle f\rangle$, $M(v)$ is the Maxwellian equilibrium 
and the rest $g$ satisfies $\langle g\rangle=0$. Here $\langle f \rangle =\int_V f(v)\dd v$, with $V=\mathbb{R}$ and we also 
use the notation $\Pi f = \langle f\rangle M$. 
The following micro-macro model is equivalent to the original model \eqref{eq:vlasovbgk}
\begin{equation}
\label{eq:micromacro_initial_E}
\left\{ 
\begin{aligned} 
 \partial_t \rho &+ \frac{1}{\varepsilon}\partial_x \langle v g \rangle = 0,\\
 \partial_t g     &+  \frac{1}{\varepsilon}(I-\Pi) \left[v \partial_x (\rho M+g) + E\partial_v (\rho M+g) \right]= -\frac{1}{\varepsilon^2}g.  
\end{aligned} 
\right. 
\end{equation}
%with $\Pi h = \langle h\rangle M$, $ \langle h\rangle =\int_V h \dd v$. 
Since $\langle vM\rangle=0$ and $\partial_v M=-vM$, the micro equation can be rewritten as%we rewrite the second term of the micro equation on $g$ as 
%\begin{eqnarray*}
%(I-\Pi) \left[v \partial_x (\rho M+g)+ E\partial_v (\rho M+g)\right] &=& vM\partial_x \rho + (I-\Pi)(v\partial_x g) 
%-vM E\rho + (I-\Pi)(E\partial_v g) \nonumber
%\end{eqnarray*}
%so that the micro equation becomes 
\be
\label{micro}
 \partial_t g     +  \frac{1}{\varepsilon}\left[ vM\partial_x \rho - vME\rho  + (I-\Pi)(v\partial_x g +E\partial_v g) \right]= -\frac{1}{\varepsilon^2}g.  
\ee
When $\varepsilon$ goes to zero, one gets from \eqref{micro}: 
$g=-\varepsilon (vM\partial_x \rho - vME\rho) + {\cal O}(\varepsilon^2)$. Then the macro equation becomes 
$$
\partial_t \rho - \partial_x \left[ \langle v^2 M\rangle \partial_x \rho -\langle v^2 M\rangle E\rho\right] = {\cal O}(\varepsilon), 
$$
which gives, using $\langle v^2 M\rangle=1$ the following drift-diffusion equation satisfied by the limit $\bar{\rho}$ 
\begin{equation}
\partial_t \bar{\rho} - \partial_{xx} \bar{\rho} + \partial_x (\bar{E}\bar{\rho})=0, \;\; \partial_x \bar{E} = \bar{\rho}-1. 
\label{eq:ddlimit}
\end{equation}
The same calculations enables to derive the micro-macro model equivalent to \eqref{eq:etrbgk} 
\begin{equation}
\label{eq:micromacro_initial}
\left\{ 
\begin{aligned} 
 \partial_t \rho &+ \frac{1}{\varepsilon}\partial_x \langle v g \rangle = 0,\\
 \partial_t g     &+  \frac{1}{\varepsilon}(I-\Pi) \left[v \partial_x (\rho M+g)\right]= -\frac{1}{\varepsilon^2}g,   
\end{aligned} 
\right. 
\end{equation}
from which we derive the corresponding asymptotic model 
$$
\partial_t \bar{\rho} - \frac{1}{3}\partial_{xx} \bar{\rho}=0, 
$$
with $M=1$, $V=\left[-1,1\right]$ and $\langle f\rangle =\frac{1}{2}\int_V f \dd v$. 

\section{First-order in time reformulation and its discretization} 
\setcounter{equation}{0}
\label{sec:1storder_0}

In this part, a first-order reformulation of the micro part is proposed, which enables 
to avoid the stiff transport term in space. The strategy is presented in the case of the equation \eqref{eq:etrbgk} 
and its corresponding micro-macro model \eqref{eq:micromacro_initial}.  

\subsection{First-order in time reformulation}
\label{sec:1storder}
We start with (\ref{eq:etrbgk}) (with periodic boundary condition in space) and consider 
the micro-macro decomposition of $f=\rho + g$ (here $M(v)=1$ for all $v\in [-1, 1]$) and the micro-macro 
model \eqref{eq:micromacro_initial}. 

First, we rewrite the micro part of \eqref{eq:micromacro_initial} as 
\be
\label{eq:expg}
\partial_t (e^{t/\varepsilon^2} g) = - \frac{e^{t/\varepsilon^2} }{\varepsilon}\mathcal{F}\left(\rho,g\right),  
\ee
where $\mathcal{F}\left(\rho,g\right)$ is given by 
\be
\label{def:F}
\mathcal{F}\left(\rho,g\right) = v\partial_x \rho + v\partial_x g - \partial_x \langle vg\rangle. 
\ee
We denote $\Delta t>0$ the time step, $t^n=n\Delta t$ with $n\in\mathbb{N}$. 
Then, a second step consists in integrating \eqref{eq:expg} on $[t^n, t^{n+1}]$ to get 
$$
g(t^{n+1}) = e^{-\Delta t/\varepsilon^2} g(t^n) - \varepsilon(1-e^{-\Delta t/\varepsilon^2}) \mathcal{F}\left(\rho(t^n),g(t^n)\right) + {\cal O}(\Delta t^2).  
$$
To derive a continuous (in time) equation, we make appear a discrete time derivative on the left-hand side 
\be
\label{micro2}
\frac{g(t^{n+1}) -g(t^n)}{\Delta t}= \frac{e^{-\Delta t/\varepsilon^2} -1}{\Delta t}g(t^n) - \varepsilon \frac{1-e^{-\Delta t/\varepsilon^2}}{\Delta t} \mathcal{F}\left(\rho(t^n),g(t^n)\right) + {\cal O}(\Delta t^2),  
\ee
which can be rewritten, up to terms of order ${\cal O}(\Delta t^2)$, as 
\begin{equation*}
%\label{micro3}
\partial_t g(t^n)= \frac{e^{-\Delta t/\varepsilon^2} -1}{\Delta t}g(t^n) - \varepsilon \frac{1-e^{-\Delta t/\varepsilon^2}}{\Delta t}\mathcal{F}\left(\rho(t^n),g(t^n)\right),~\forall n.   
\end{equation*}
We finally obtain the first-order reformulation of \eqref{eq:etrbgk} 
\begin{align}
 \partial_t \rho &+ \frac{1}{\varepsilon}\partial_x \langle v g \rangle = 0,\label{eq:mm_macro}\\
\partial_t g&= \frac{e^{-\Delta t/\varepsilon^2} -1}{\Delta t}g - \varepsilon \frac{1-e^{-\Delta t/\varepsilon^2}}{\Delta t}
\mathcal{F}\left(\rho,g\right), % \left[ v\partial_x \rho + v\partial_x g - \partial_x \langle vg\rangle   \right].
\label{eq:mm_micro}
\end{align}
with $\mathcal{F}\left(\rho,g\right)$ given by \eqref{def:F}. 
We remark that the micro equation does not contain any stiff term and then has a suitable form for a numerical 
discretization using particles. 
Moreover, this first-order reformulation satisfies the following properties.
\begin{itemize}
\item Consistency: For all fixed $\varepsilon >0$,  equation (\ref{eq:mm_micro}) is consistent with the initial micro equation (\ref{micro}) as $\Delta t$ goes to zero.
\item Asymptotic behaviour: For all fixed  $\Delta t >0$, as $\varepsilon$ goes to zero, we get from (\ref{eq:mm_micro}) $g=-\varepsilon v \partial_x \rho+ O(\varepsilon^2)$, which injected in the macro equation (\ref{eq:mm_macro}) provides the limit model 
\eqref{eq:diff}. 
\end{itemize}

%In the following subsection, we are interested in the discretization of system (\ref{eq:mm_macro})-(\ref{eq:mm_micro}).

%################################################################
%SECTION 2
%################################################################
%\section{Discretization}\label{sec:discretization}
%\setcounter{equation}{0}

\subsection{Lagrangian discretization}\label{ssec:pic}

This subsection is devoted to the derivation of an AP-particle based numerical scheme 
for  (\ref{eq:mm_macro})-(\ref{eq:mm_micro}).

\paragraph{Explicit in time discretization\\} 
We propose now a Lagrangian discretization of (\ref{eq:mm_micro}). 
More precisely, we adopt a particle method, see \cite{birdsall}, and consider a set of $N_p\in\mathbb{N}$ 
macro particles. The position of particle $k$, $1\leq k\leq N_p$, is denoted by 
$x_k(t)\in\Omega= \left[0,L_x\right]$, with $L_x>0$, its velocity by $v_k(t)\in V=\left[-1,1\right]$ and 
its weight by $\omega_k(t)\in\mathbb{R}$. Let $L_v=|V|=2$. 
The function $g$ is then assumed to be of the form 
\begin{equation}
g\left(t,x,v\right)=\sum_{k=1}^{N_p}\omega_k(t)\delta(x-x_k(t))\delta(v-v_k(t)),
\label{eq:diracmasses}
\end{equation}
where $\delta$ denotes the Dirac mass function. 
Weights $\omega_k(t)$ are related to the distribution function $g$ through
\begin{equation}\label{eq:poids}
\omega_k(t)=g(t,x_k(t),v_k(t))\frac{L_xL_v}{N_p}.
\end{equation}
Initially, particles are randomly distributed in the phase-space domain $\left[0,L_x\right]\times V$ 
and their weights are computed following (\ref{eq:poids}). 

The density $\rho$ is computed on a uniform spatial grid defined by $\mbox{x}_i=i\Delta x$, $i=0,\dots,N_x$, $N_x\in\mathbb{N}^\star$ and $\Delta x=L_x/N_x$. 
We denote by $\rho_i^n$ the approximation at time $t^n=n\Delta t$ and position $\mbox{x}_i$ of $\rho(t^n,\mbox{x}_i)$, 
with $\Delta t>0$ the time step. Moreover, $g^n(x,v)\approx g(t^n, x, v)$, $x_k^n\approx x_k(t^n)$, $v_k^n\approx v_k(t^n)$ 
and $w_k^n\approx w_k(t^n)$. 
Let us remark that $\dot{v}_k=0$, so that the velocities $v_k(t)$ are constant in time and we will note $v_k^n=v_k^0=:v_k$ for all $n$. 

Our goal is then to extend the particle discretization of \cite{ccl} to diffusion scaling. To that 
purpose, we exploit the reformulation (\ref{eq:mm_micro}). As already said in \cite{ccl}, we have 
to use a splitting procedure between the transport part and the source part. Then, the 
(first-order) splitting writes 
\begin{itemize}
\item start with an initial repartition of the $N_p$ particles $(x_k^0, v_k^0)$, 
with $\omega_k^0=g(t=0, x_k^0, v_k^0)L_x L_v/N_p$,
\item solve the transport part
\begin{equation*}
\partial_t g = \varepsilon \frac{1-e^{-\Delta t/\varepsilon^2}}{\Delta t}  v\partial_x g,
%\label{eq:split1}
\end{equation*} 
with the (non stiff) characteristics 
\be
\label{carx}
\dot{x}_k =  \varepsilon \frac{1-e^{-\Delta t/\varepsilon^2}}{\Delta t}  v_k,
\ee
\item solve the source part
\begin{equation*}
\partial_t g =  \frac{e^{-\Delta t/\varepsilon^2} -1}{\Delta t}g- \varepsilon\frac{1-e^{-\Delta t/\varepsilon^2}}{\Delta t} \left[ v\partial_x \rho  - \partial_x \langle vg\rangle \right],
%\label{eq:split2}
\end{equation*}  
using the equation satisfied by the weights 
\be
\label{weight}
\dot{\omega}_k = \frac{e^{-\Delta t/\varepsilon^2} -1}{\Delta t}\omega_k - \varepsilon \frac{1-e^{-\Delta t/\varepsilon^2}}{\Delta t} \left[ v_k (\partial_x \rho(x_k) - \partial_x \langle vg \rangle (x_k) \right]\frac{L_x L_v}{N_p}. 
\ee  
\end{itemize}
%To do that, the transport part (\ref{eq:split1}) is solved with the (non stiff) characteristics 
%\be
%\label{carx}
%\dot{x}_k =  \varepsilon \frac{(1-e^{-\Delta t/\varepsilon^2})}{\Delta t}  v_k.
%\ee
%The source part (\ref{eq:split2}) is solved using the equation satisfied by the weights 
%\be
%\label{weight}
%\dot{\omega}_k = \frac{(e^{-\Delta t/\varepsilon^2} -1)}{\Delta t}\omega_k^n - \varepsilon \frac{(1-e^{-\Delta t/\varepsilon^2})}{\Delta t} \left[ v_k M(v_k)(\partial_x \rho(x_k) - \partial_x \langle vg \rangle (x_k) M(v_k) \right]. 
%\ee

Now, we detail the time discretization of the two steps. First, (\ref{carx}) is approximated by 
a simple forward Euler scheme
\be
\label{carxd}
x_k^{n+1} = x_k^n + \varepsilon (1-e^{-\Delta t/\varepsilon^2})v_k. 
\ee
Second, we compute the last term in \eqref{weight}. The term $\langle v g \rangle$ is 
approximated on the spatial grid $\mbox{x}_i$ using 
%these new positions $x_k^{n+1}$ and the approximation (called a deposition procedure) 
\be
\label{momg}
\langle v g \rangle(\mbox{x}_i) \approx   \sum_{k=1}^{N_{p}} \omega^n_k B_\ell(\mbox{x}_i-x_k^{n+1}) v_k, 
\ee
%where $x_k^n$, $v_k^n=v_k^0$, $\omega_k^n$ denote the position, the (in fact constant) velocity  and the weight of the $k$-th particle at time $t^n$, and 
where $B_\ell \geq 0$ is a B-spline function of order $\ell$: 
\begin{equation}
\label{bspline}
B_\ell(x)=(B_0 * B_{\ell -1})(x), \;\; \mbox{ with } \;\; 
B_0(x)=
\left\{
\begin{array}{llcc}
\frac{1}{\Delta x}  & \mbox{ if } |x|<\Delta x/2,  \\ 
0  & \mbox{ else}. 
\end{array}
\right.
\end{equation}
We then approximate the equation on the weights (\ref{weight}) using a first-order explicit integrator  
\be
\label{weightd}
\omega^{n+1}_k = e^{-\Delta t/\varepsilon^2} \omega_k^n - \varepsilon (1-e^{-\Delta t/\varepsilon^2}) \left[  \alpha_k^n + \beta_k^n\right],  
\ee
with 
\be
\label{poids_rhs}
\alpha_k^n= v_k \partial_x \rho^n(x_k^{n+1}) \frac{L_x L_v}{N_p}~~~\textrm{and}~~~\beta_k^n=- \partial_x \langle vg \rangle (x_k^{n+1}) \frac{L_x L_v}{N_p}.
\ee
To compute $\alpha_k^n$ (resp. $\beta_k^n$), since $\rho^n$ (resp. $\langle v g \rangle$) 
is known on the spatial grid, we 
approximate $\partial_x \rho^n$ (resp. $\partial_x\langle v g \rangle$) by centered 
finite differences and evaluate at $x_k^{n+1}$ using an interpolation with B-spline functions, for example
$$
\partial_x \rho^n(x_k^{n+1})\approx\sum_{i=1}^{N_x}\frac{\rho_{i+1}^n-\rho_{i-1}^n}{2\Delta x}B_\ell(\mbox{x}_i-x_k^{n+1}).
$$ 
%\\
%\noindent To compute $\beta_k^n$, since $\langle v g^\star \rangle$ is known on the spatial grid (with (\ref{momg})), we 
%approximate $\partial_x\langle v g^n \rangle$ by finite differences and evaluate at $x=x_k$ using an interpolation. 
Finally, the macro equation (\ref{eq:mm_macro}) is advanced through 
\begin{equation}
\rho_i^{n+1} = \rho_i^n -\frac{\Delta t}{\varepsilon} \frac{\langle vg^{n+1}\rangle_{i+1}-\langle vg^{n+1}\rangle_{i-1}}{2\Delta x}, 
\label{macrod0}
\end{equation}
where $\langle v g^{n+1}\rangle_i$ is computed using \eqref{momg}. 
Then, we have the following proposition.

\begin{prop}
The scheme given by (\ref{carxd})-(\ref{weightd})-(\ref{macrod0}) enjoys the AP property, \textit{i.e.} it satisfies the following properties 
%is a first-order AP scheme 
%for system (\ref{eq:mm_macro})-(\ref{eq:mm_micro})
\begin{itemize}
\item for fixed $\varepsilon>0$, the scheme is a first-order (in time) approximation of the original model (\ref{eq:etrbgk}),%  (\ref{eq:mm_macro})-(\ref{eq:mm_micro}), 
\item for fixed $\Delta t>0$, the scheme degenerates into an explicit first-order (in time) scheme of \eqref{eq:diff}. 
\end{itemize}
%discretization of the diffusion equation $\partial_t\rho-\partial_{xx}\rho=0$ when $\varepsilon\to 0$. A necessary stability condition is $\Delta t={\cal O}(\Delta x^2)$, coming from the diffusion term.
\label{prop:lag_1st_exp}
\end{prop}
\begin{proof}
The consistency follows directly from standard approximation. For the asymptotic behavior, 
when $\varepsilon$ goes to zero, we get 
$\omega_k^{n+1} = -\varepsilon \alpha_k^n+{\cal O}(\varepsilon^2)$ (since $\omega_k^n={\cal O}(\varepsilon)$ 
$\forall n\geq 1$). 
Computing the momentum of $g^{n+1}$ means that we use (\ref{momg}) with $g^{n+1}$, or in 
the limit regime 
\begin{eqnarray*}
\langle v g^{n+1} \rangle_i &\approx  & -\varepsilon \sum_{k=1}^{N_{p}} \alpha^n_k B_\ell(\mbox{x}_i-x_k^{n+1}) v_k^{n+1} + {\cal O}(\varepsilon^2),  \nonumber\\
&\approx & -\varepsilon\left[ \langle v^2 \rangle  \partial_x \rho^n \right]|_{x=\mbox{x}_i}  + {\cal O}(\varepsilon^2) \nonumber\\
&\approx &-\varepsilon  \frac{1}{3} \partial_x \rho_i^n + {\cal O}(\varepsilon^2).
\end{eqnarray*}
Injecting in the macro equation (\ref{eq:mm_macro}) then leads to a consistent discretization of \eqref{eq:diff}. 
\end{proof}
\begin{remark}[Preservation of the micro-macro structure.] As detailed in \cite{ccl}, we have to correct the particle' weights in order to preserve at the numerical level the micro-macro structure. Indeed, the micro-macro decomposition technique uses the zero-mean property $\langle g\rangle=0$. But nothing guarantees that this property is satisfied at the discrete level (on the weights $\omega_k$). That is why we have to correct the weights, by applying a discrete approximation of the operator $\left(I-\Pi\right)$ to each weight $\omega_k$, which is consistent with the continuous model.

We do not give the details of this correction (called \textit{projection step} in following algorithms) and refer the reader to \cite{ccl}. %But we would like to remark that it consists in substracting to each weight $\omega_k$ a quantity computed
\end{remark}

\paragraph{Implicit time discretization\\} 
In this part, we want to make the previous scheme  (\ref{carxd})-(\ref{weightd})-(\ref{macrod0}) degenerate 
into an implicit time discretization of \eqref{eq:diff} (as done in \cite{lemou-note, cl}). 
To do so, we decompose (\ref{weightd}) into two parts so that the macro flux becomes 
\begin{eqnarray*}
%w^{n+1}= - \varepsilon (1-e^{-\Delta t/\varepsilon^2})  \alpha^n_k  +h_i^n, \mbox{ with } h_i^n= e^{-\Delta t/\varepsilon^2} \omega^n_k -\varepsilon (1-e^{-\Delta t/\varepsilon^2}) \beta^n_k.
\langle v g^{n+1} \rangle_i &=  &  - \varepsilon (1-e^{-\Delta t/\varepsilon^2}) \sum_{k=1}^{N_{p}} \alpha^n_k B_\ell(\mbox{x}_i-x_k^{n+1}) v_k +h_i^n,\nonumber
\end{eqnarray*}
with $B_\ell$ given by \eqref{bspline} and 
%\be
%\label{h}
$$
h_i^n= e^{-\Delta t/\varepsilon^2}\sum_{k=1}^{N_{p}} \omega^n_k B_\ell(\mbox{x}_i-x_k^{n+1}) v_k 
-\varepsilon (1-e^{-\Delta t/\varepsilon^2}) \sum_{k=1}^{N_{p}} \beta^n_k B_\ell(\mbox{x}_i-x_k^{n+1}) v_k.
$$
%\ee
Since $\alpha_k^n$ is the weight of $v\partial_x \rho^n$ (see \eqref{poids_rhs}) and $\langle v^2\rangle=1/3$, we can write  
\begin{eqnarray*}
\langle v g^{n+1} \rangle_i &\approx & - \varepsilon (1-e^{-\Delta t/\varepsilon^2}) \frac{1}{3}\partial_x \rho^n_i +h_i^n, 
\end{eqnarray*}
so that the macro scheme can be 
\begin{eqnarray}
\rho_i^{n+1} &=& \rho_i^n + \Delta t   (1-e^{-\Delta t/\varepsilon^2}) \frac{1}{3}\frac{ \rho^n_{i+1}-2\rho^n_i+\rho^n_{i-1}  }{\Delta x^2}-\frac{\Delta t}{\varepsilon}\frac{h_{i+1}^n-h_{i-1}^n}{2\Delta x}. 
\label{eq:lag_manoeuvre}
\end{eqnarray}
%The asymptotic model has been explicitly written in the macro part and since $h_{i}^n= {\cal O}(\varepsilon^2)$ as $\varepsilon$ goes to zero after two iterations, the asymptotic preserving property is ensured. 
We can go further by considering now the diffusion term implicit in time to get 
\begin{eqnarray}
\rho_i^{n+1} &=& \rho_i^n + \Delta t   (1-e^{-\Delta t/\varepsilon^2}) \frac{1}{3}\frac{ \rho^{n+1}_{i+1}-2\rho^{n+1}_i+\rho^{n+1}_{i-1}  }{\Delta x^2}-\frac{\Delta t}{\varepsilon}\frac{h_{i+1}^n-h_{i-1}^n}{2\Delta x}. 
\label{macrod0_cn}
\end{eqnarray}
%Hence, when $\varepsilon$ goes to zero, one finally gets the following first-order scheme for the asymptotic model, 
%where the diffusion term is implicit 
%Since $\omega_k^n$ is of order $\varepsilon$ for $n\geq 1$, 
%$$
%\rho^{n+1}_i = \rho^n_i + \Delta t  \frac{ \rho^{n+1}_{i+1}-2\rho^{n+1}_i+\rho^{n+1}_{i-1}  }{\Delta x^2}.  
%$$
We can write the following proposition. 
\begin{prop}
The scheme given by (\ref{carxd})-(\ref{weightd})-(\ref{macrod0_cn}) enjoys the AP property, \textit{i.e.} it satisfies the following properties 
%is a first-order AP scheme 
%for system (\ref{eq:mm_macro})-(\ref{eq:mm_micro})
\begin{itemize}
\item for fixed $\varepsilon>0$, the scheme is a first-order (in time) approximation of the original model (\ref{eq:etrbgk}),% (\ref{eq:mm_macro})-(\ref{eq:mm_micro}), 
\item for fixed $\Delta t>0$, the scheme degenerates into an implicit first-order (in time) scheme of \eqref{eq:diff}. 
\end{itemize}
%discretization of the diffusion equation $\partial_t\rho-\partial_{xx}\rho=0$ when $\varepsilon\to 0$. A necessary stability condition is $\Delta t={\cal O}(\Delta x^2)$, coming from the diffusion term.
\label{prop:lag_1st_imp}
\end{prop}

\begin{remark}
Instead of an implicit scheme, it is possible to use a Crank-Nicolson method.
\end{remark}

%The scheme detailed in Algorithm \ref{algo:lag_1st} for system (\ref{eq:mm_macro})-(\ref{eq:mm_micro}) enjoys the Asymptotic Preserving property and is free from the usual diffusion condition on the time step.
%\end{prop}

The scheme is finally summarized in the following algorithm.
\begin{algo}~~ 
\begin{itemize}
\item  Initialize $(x_k^0, v_k^0)$, $\omega_k^0$ and $\rho_i^0$.

At each time step:
\item 1) Advance micro part: 
\begin{itemize}
\item advance the characteristics with (\ref{carxd}),
\item compute $\langle v g\rangle$ with (\ref{momg}),
\item advance the equation on the weights with (\ref{weightd}).
\end{itemize}
\item 2) Projection step: compute $(I-\Pi)g^{n+1}$ using \cite{ccl}.
\item 3) Advance macro part: 
\begin{itemize}
\item compute $\langle v g^{n+1}\rangle$ with (\ref{momg}),
\item compute $\rho^{n+1}$ with (\ref{macrod0_cn}).
\end{itemize}
\end{itemize}
\label{algo:lag_1st}
\end{algo}

\section{Second-order in time reformulation and its discretization}\label{sec:2ndorder_0}
\setcounter{equation}{0}

This section is devoted to the derivation of a second-order scheme for the micro-macro system (\ref{eq:micromacro_initial}). As for the first-order scheme, we will first reformulate 
the microscopic equation (\ref{micro}) in order to suppress stiff terms (see Subsection \ref{sec:1storder} 
for the first-order case), and then discretize the obtained micro-macro model to get an AP 
efficient numerical scheme (see Subsection \ref{ssec:pic} for the first-order case). 

\subsection{Second-order reformulation}\label{sec:2ndorder}
%Starting from \eqref{eq:expg}-\eqref{def:F}, a second-order approximation 
%\subsection{Rewriting of the microscopic equation% (\ref{micro})
%We aim at rewriting the microscopic equation (\ref{micro}) thanks to an exponential scheme in order to eliminate stiff terms. Multiplying (\ref{micro}) by $e^{t/\varepsilon^2}$ gives
Let start from the following (equivalent) reformulation of the micro part of \eqref{eq:micromacro_initial} 
$$
\partial_t\left(e^{t/\varepsilon^2}g\right)=-\frac{e^{t/\varepsilon^2}}{\varepsilon}\mathcal{F}\left(\rho(t),g(t)\right),
$$
where $\mathcal{F}\left(\rho,g\right)$ is defined by \eqref{def:F}. 
We now integrate with respect to $t\in [t^n, t^{n+1}]$ and use a second-order mid-point quadrature 
%$$
%e^{t^{n+1}/\varepsilon^2}g(t^{n+1})-e^{t^{n}/\varepsilon^2}g(t^{n})=-\int_{t^n}^{t^{n+1}}\frac{e^{s/\varepsilon^2}}{\varepsilon}\mathcal{F}\left(\rho(s),g(s)\right)\dd s.
%$$
%The key point is to use a midpoint method to approximate the right-hand side. Then we divide by $e^{t^{n+1}/\varepsilon^2}$, which leads to
$$
g(t^{n+1})=e^{-\Delta t/\varepsilon^2}g(t^{n})-\frac{\Delta te^{-\Delta t/2\varepsilon^2}}{\varepsilon}\mathcal{F}\left(\rho(t^{n+1/2}),g(t^{n+1/2})\right)+\mathcal{O}\left(\Delta t^3\right).
$$
To derive a continuous (in time) equation, we make appear a discrete time derivative on the left-hand side 
$$
\frac{g(t^{n+1})-g(t^{n})}{\Delta t}
=\frac{e^{-\Delta t/\varepsilon^2}-1}{\Delta t}g(t^{n})-\frac{e^{-\Delta t/2\varepsilon^2}}{\varepsilon}\mathcal{F}\left(\rho(t^{n+1/2}),g(t^{n+1/2})\right)+\mathcal{O}\left(\Delta t^2\right).
$$
We now look for a continuous (in time) equation for which the previous relation is a second numerical scheme.  
To do so, we perform Taylor expansions of the different terms  at $t^{n+1/2}$ 
%$\frac{e^{-\Delta t/\varepsilon^2}-1}{\Delta t}g(t^{n})$ at $t^{n+1/2}$ and since $\frac{e^{-\Delta t/\varepsilon^2}-1}{\Delta t}$ is of order 1 (with respect to $\Delta t$), we get
%\begin{equation}\label{eq:approx_o2_1}
%\begin{split}
%\partial_t g(t^{n+1/2})
%=\frac{e^{-\Delta t/\varepsilon^2}-1}{\Delta t}\left(g(t^{n+1/2})-\frac{\Delta t}{2}\partial_tg(t^{n+1/2})\right)~~~~~~~~~~~~~~~~~~~~\\
%-\frac{e^{-\Delta t/2\varepsilon^2}}{\varepsilon}\mathcal{F}\left(\rho(t^{n+1/2}),g(t^{n+1/2}))\right)
%+\mathcal{O}\left(\Delta t^2\right).
%\end{split}
%\end{equation}
\begin{equation}\label{eq:approx_o2_1}
\partial_t g(t^{n+1/2})
\!=\!\frac{e^{-\Delta t/\varepsilon^2}\!\!-\!1}{\Delta t} \!\! \left(g(t^{n+1/2})-\frac{\Delta t}{2}\partial_tg(t^{n+1/2})\right)
-\frac{e^{-\Delta t/2\varepsilon^2}}{\varepsilon}\mathcal{F}\!\left(\rho(t^{n+1/2}),g(t^{n+1/2})\right)
+\mathcal{O}\left(\Delta t^2\right)\!.
\end{equation}
%Since $\frac{g(t^{n+1})-g(t^{n})}{\Delta t}$ is a second-order approximation of $\partial_tg(t^{n+1/2})$, (\ref{eq:approx_o2_1}) rewrites
%We finally get 
%$$
%\partial_tg(t^{n+1/2})=\frac{2}{\Delta t}\frac{e^{-\Delta t/\varepsilon^2}-1}{e^{-\Delta t/\varepsilon^2}+1}g(t^{n+1/2})-\frac{2}{\varepsilon}\frac{e^{-\Delta t/2\varepsilon^2}}{e^{-\Delta t/\varepsilon^2}+1}\mathcal{F}\left(\rho(t^{n+1/2}),g(t^{n+1/2}))\right)+\mathcal{O}\left(\Delta t^2\right).
%$$
Finally, the microscopic equation of \eqref{eq:micromacro_initial} is reformulated up to the second-order by
$$
\partial_t g=\frac{2}{\Delta t}\frac{e^{-\Delta t/\varepsilon^2}-1}{e^{-\Delta t/\varepsilon^2}+1}g-\frac{2}{\varepsilon}\frac{e^{-\Delta t/2\varepsilon^2}}{e^{-\Delta t/\varepsilon^2}+1}\mathcal{F}\left(\rho,g\right),
$$
and we can now consider the second-order reformulated micro-macro system
\begin{align}
 \partial_t \rho &+ \frac{1}{\varepsilon}\partial_x \langle v g \rangle = 0,\label{eq:mm_macro_2nd}\\
\partial_tg&=\frac{2}{\Delta t}\frac{e^{-\Delta t/\varepsilon^2}-1}{e^{-\Delta t/\varepsilon^2}+1}g-\frac{2}{\varepsilon}\frac{e^{-\Delta t/2\varepsilon^2}}{e^{-\Delta t/\varepsilon^2}+1} \left[ v\partial_x \rho + v\partial_x g - \partial_x \langle vg\rangle  \right].
\label{eq:mm_micro_2nd}
\end{align}

\subsection{Time discretization}\label{subsec:ordre2_discr}

We are now interested in the construction of an AP scheme for system (\ref{eq:mm_macro_2nd})-(\ref{eq:mm_micro_2nd}), based on a second-order Runge-Kutta (RK2) method for the time discretization and a Lagrangian method for the phase space discretization of the micro part. 

%\paragraph{Second-order semi-discretization in time.} 

The RK2 is based on a prediction step on $\Delta t/2$ (first-order) and a correction step on $\Delta t$. 
Then, a second-order (in time) scheme for (\ref{eq:mm_macro_2nd})-(\ref{eq:mm_micro_2nd}) would read 
\begin{align}
\mbox{Prediction step on $\Delta t/2$}&\nonumber\\
\begin{split}
g^{n+1/2}=&g^n+\frac{e^{-\Delta t/\varepsilon^2}-1}{e^{-\Delta t/\varepsilon^2}+1}\; g^n-\frac{\Delta t}{\varepsilon}\frac{e^{-\Delta t/2\varepsilon^2}}{e^{-\Delta t/\varepsilon^2}+1} \mathcal{F}\left(\rho^n,g^n\right), %\left[ v\partial_x \rho^n + v\partial_x g^n - \partial_x \langle vg^n\rangle \right],
\end{split}\label{gnpdemi_0}\\
\rho^{n+1/2} =&\rho^n- \frac{\Delta t}{2\varepsilon}\partial_x \langle v g^{n+1/2} \rangle,\label{rhonpdemi_0}\\
\mbox{Correction step on $\Delta t$}\hphantom{/2}&\nonumber\\
\begin{split}
g^{n+1}=&g^n+2\frac{e^{-\Delta t/\varepsilon^2}-1}{e^{-\Delta t/\varepsilon^2}+1}\; \widetilde{g}-\frac{2\Delta t}{\varepsilon}\frac{e^{-\Delta t/2\varepsilon^2}}{e^{-\Delta t/\varepsilon^2}+1} \mathcal{F}\left(\rho^{n+1/2},g^{n+1/2}\right), %\left[ v\partial_x \rho^{n+1/2} + v\partial_x g^{n+1/2} \right. \\
%& \left. - \partial_x \langle vg^{n+1/2}\rangle  \right],
\end{split}\label{gnp1_0}\\
\rho^{n+1} =&\rho^n- \frac{\Delta t}{\varepsilon}\partial_x \langle v g^{n+1/2} \rangle.\label{rhonp1_0}
\end{align}
%where the space and velocity dependences are not written for simplifying the notations.
%Firstly, let us remark that we can take $\partial_x \langle v g^{n+1/2} \rangle$ instead of $\partial_x \langle v g^{n} \rangle$ in (\ref{rhonpdemi_0}) since this step is a first-order approximation. As $\partial_x \langle v g^{n+1/2} \rangle$ has to be computed for the last step (in (\ref{rhonp1_0})), we use it in (\ref{rhonpdemi_0}) and never compute $\partial_x \langle v g^{n} \rangle$ for saving computational cost.
We still have to fix $\widetilde{g}$ in \eqref{gnpdemi_0} in order to get a second-order scheme and to 
ensure the convergence of $g^{n+1}$ to zero as $\varepsilon$ goes to zero. 
It turns out the choice $\widetilde{g}=\frac{g^n+g^{n+1}}{2}$ ensures the two conditions. 
Indeed, we get for the correction step of the micro part 
%Secondly, in order to get a consistent scheme at the limit $\varepsilon\to 0$, we approximate $2\frac{e^{-\Delta t/\varepsilon^2}-1}{e^{-\Delta t/\varepsilon^2}+1}g^{n+1/2}$ by $2\frac{e^{-\Delta t/\varepsilon^2}-1}{e^{-\Delta t/\varepsilon^2}+1}\frac{g^n+g^{n+1}}{2}$ in (\ref{gnp1_0}).
%Finally, the second-order time discretization we obtain is 
%\begin{align}
%\begin{split}
%g^{n+1/2}=&\frac{2e^{-\Delta t/\varepsilon^2}}{e^{-\Delta t/\varepsilon^2}+1}g^n-\frac{\Delta t}{\varepsilon}\frac{e^{-\Delta t/2\varepsilon^2}}{e^{-\Delta t/\varepsilon^2}+1} \mathcal{F}\left(\rho^{n},g^{n}\right), %\left[ vM\partial_x \rho^n + v\partial_x g^n - \partial_x \langle vg^n\rangle M \right],
%\end{split}\label{gnpdemi}\\
%\rho^{n+1/2} =&\rho^n- \frac{\Delta t}{2\varepsilon}\partial_x \langle v g^{n+1/2} \rangle,\label{rhonpdemi}\\
%\begin{split}
$$
g^{n+1}=e^{-\Delta t/\varepsilon^2}g^n-\frac{\Delta t}{\varepsilon}e^{-\Delta t/2\varepsilon^2} \mathcal{F}\left(\rho^{n+1/2},g^{n+1/2}\right).  %\left[ vM\partial_x \rho^{n+1/2} + v\partial_x g^{n+1/2} - \partial_x \langle vg^{n+1/2}\rangle M\right],
$$
%\end{split}\label{gnp1}\\
%\rho^{n+1} =&\rho^n- \frac{\Delta t}{\varepsilon}\partial_x \langle v g^{n+1/2} \rangle.\label{rhonp1}
%\end{align}

%Equipped with this time scheme, we can derive a fully discretized scheme of Lagrangian (resp. Eulerian) type in space, as in Subsection \ref{ssec:pic} (resp. Appendix \ref{subsec:eulerian}).
%The next paragraph is devoted to the full discretization of Lagrangian type. The Eulerian one being more classical, we let the interested reader adapt the computations of Appendix \ref{subsec:eulerian}.  

%\paragraph{Full discretization of Lagrangian scheme.}

Up to now, the micro part converges exponentially fast to zero (when $\varepsilon$ goes to zero), 
so that the asymptotic behavior of the scheme is $\rho^{n+1} = \rho^n$. Hence, the last but not the least 
step consists in modifying the macro part to capture the correct asymptotic limit. 

As done in Section \ref{sec:1storder_0}, we will modify the discretization in order to make 
appear the diffusion term directly in the macro part \eqref{rhonp1_0}. However, the modification 
has to be of order $\Delta t^3$ for a fixed $\varepsilon>0$ to not spoil the second-order accuracy of the scheme. 
The correction we propose consists in computing $\rho_i^{n+1}$ using a Crank-Nicolson 
scheme for the diffusion weighted which is of order ${\cal O}(\Delta t^3)$ for fixed $\varepsilon>0$ 
and which degenerates to $\Delta t$ when $\varepsilon$ goes to zero 
\begin{equation}
\rho_i^{n+1}=\rho_i^n- \frac{\Delta t}{\varepsilon}\partial_x \langle v g^{n+1/2} \rangle_i+\Delta t(1-e^{-\Delta t/\varepsilon^2})^2\frac{1}{3}\partial_{xx}\left(\frac{\rho_i^{n+1}+\rho_i^n}{2}\right).
\label{eq:rho_ordre2_man}
\end{equation}

%Now, we can implicit the diffusion term (more precisely use a Crank-Nicolson method) by replacing $\partial_{xx}\rho_i^{n+1/2}$ by $\partial_{xx}\left(\frac{\rho_i^{n+1}+\rho_i^n}{2}\right)$. We obtain
%\begin{equation}
%\rho_i^{n+1}=\rho_i^n- \frac{\Delta t}{\varepsilon}\partial_x \langle v g^{n+1/2} \rangle_i+\Delta t(1-e^{-\Delta t/\varepsilon^2})^2\partial_{xx}\left(\frac{\rho_i^{n+1}+\rho_i^n}{2}\right).
%\label{eq:rho_corr}
%\end{equation}

\subsection{Lagrangian discretization}\label{ssec:pic_o2}
We consider the same notations as in Subsection \ref{ssec:pic} 
and detail here the Lagrangian discretization of the micro-macro system (\ref{eq:mm_macro_2nd})-(\ref{eq:mm_micro_2nd}). In the prediction step \eqref{gnpdemi_0}, we compute $x_k^{n+1/2}$ with 
a forward Euler integrator 
\begin{equation}
x_k^{n+1/2}=x_k^n+\frac{\Delta t}{\varepsilon}\frac{e^{-\Delta t/2\varepsilon^2}}{e^{-\Delta t/\varepsilon^2}+1}v_k,
\label{eq:cara_pred}
\end{equation}
and advance the weights with
\begin{equation}
\begin{split}
w_k^{n+1/2}=\frac{2e^{-\Delta t/\varepsilon^2}}{e^{-\Delta t/\varepsilon^2}+1}w_k^n-\frac{\Delta t}{\varepsilon}\frac{e^{-\Delta t/2\varepsilon^2}}{e^{-\Delta t/\varepsilon^2}+1}\left[ v_k \partial_x \rho^n(x_k^n)- \partial_x \langle v_k g^n(x_k^n)\rangle  \right]\frac{L_xL_v}{N_p}. 
\end{split}\label{eq:poids_pred}
\end{equation}
We end this prediction step by computing the flux $\langle v g^{n+1/2}\rangle$ with \eqref{momg} to get the density 
\begin{equation}
\rho_i^{n+1/2}=\rho_i^n- \frac{\Delta t}{2\varepsilon}\partial_x \langle v g^{n+1/2} \rangle_i.
\label{eq:rho_pred}
\end{equation}
Now in the correction step, we compute the position at $t^{n+1}$ with
\begin{equation}
x_k^{n+1}=x_k^n+\frac{\Delta t}{\varepsilon}e^{-\Delta t/2\varepsilon^2}v_k, 
\label{eq:cara_corr}
\end{equation}
then the weights are given by 
\begin{equation}
w_k^{n+1}=e^{-\Delta t/\varepsilon^2}w_k^n-\frac{\Delta t}{\varepsilon}e^{-\Delta t/2\varepsilon^2}\left[ v_k \partial_x \rho^{n+1/2}(x_k^{n+1/2}) - \partial_x \langle v_kg^{n+1/2}(x_k^{n+1/2})\rangle \right]\frac{L_xL_v}{N_p}.
\label{eq:poids_corr}
\end{equation}
%As for the first-order scheme, we can modify our discretization in order to make appear the limit directly in the macroscopic equation on $\rho_i^{n+1}$.  But here, the correction has to be of order $\Delta t^3$ for a fixed $\varepsilon>0$. The correction we propose consists in computing $\rho_i^{n+1}$ following  
%\begin{equation}
%\rho_i^{n+1}=\rho_i^n- \frac{\Delta t}{\varepsilon}\partial_x \langle v g^{n+1/2} \rangle_i+\Delta t(1-e^{-\Delta t/\varepsilon^2})^2\partial_{xx}\rho_i^{n+1/2}.
%\label{eq:rho_ordre2_man}
%\end{equation}
Now, using \eqref{eq:rho_ordre2_man} in the last step, we have  
%we can implicit the diffusion term (more precisely use a Crank-Nicolson method) by replacing $\partial_{xx}\rho_i^{n+1/2}$ by $\partial_{xx}\left(\frac{\rho_i^{n+1}+\rho_i^n}{2}\right)$. We obtain
\begin{equation}
\rho_i^{n+1}=\rho_i^n- \frac{\Delta t}{\varepsilon}\partial_x \langle v g^{n+1/2} \rangle_i+\Delta t(1-e^{-\Delta t/\varepsilon^2})^2\frac{1}{3}\partial_{xx}\left(\frac{\rho_i^{n+1}+\rho_i^n}{2}\right), 
\label{eq:rho_corr}
\end{equation}
where $\langle v g^{n+1/2} \rangle_i$ is computed using \eqref{momg}. 

We finally have the following result.
\begin{prop}
The scheme given by \eqref{eq:cara_pred}-\eqref{eq:poids_pred}-\eqref{eq:rho_pred}-\eqref{eq:cara_corr}-\eqref{eq:poids_corr}-\eqref{eq:rho_corr} enjoys the AP property, \textit{i.e.} it satisfies the following properties 
\begin{itemize}
\item for fixed $\varepsilon>0$, the scheme is a second-order (in time) approximation of the original model (\ref{eq:etrbgk}), %(\ref{eq:mm_macro_2nd})-(\ref{eq:mm_micro_2nd}), 
\item for fixed $\Delta t>0$, the scheme degenerates into an implicit second-order (in time) scheme of \eqref{eq:diff}. 
\end{itemize}
%discretization of the diffusion equation $\partial_t\rho-\partial_{xx}\rho=0$ when $\varepsilon\to 0$. A necessary stability condition is $\Delta t={\cal O}(\Delta x^2)$, coming from the diffusion term.
\label{prop:lag_2st}
\end{prop}

\begin{proof}
When $\varepsilon\to 0$, we get from (\ref{eq:poids_pred}) $w_k^{n+1/2}\to 0$ exponentially fast and then $\langle v g^{n+1/2} \rangle_i\to 0$. By injecting it in the macro equation (\ref{eq:rho_corr}), we have at the limit $\rho_i^{n+1}=\rho_i^n+\frac{\Delta t}{3}\partial_{xx}\left(\frac{\rho_i^{n+1}+\rho_i^n}{2}\right)$, which is a Crank-Nicolson discretization of the diffusion equation \eqref{eq:diff}.% petite preuve sur le caractere AP a faire... ($g$ tend vers $0$ et donc $\rho$ solution de CN)
\end{proof}

\begin{remark}
Let us emphasize that the moments $\langle \cdot \rangle$ have to be computed with B-spline functions of order $\ell\geq 1$ in order to obtain a second-order in time scheme. Taking $\ell=0$ would lead to space discontinuities preventing the time scheme to be of second-order.
\end{remark}

\begin{remark}
The projection step being by construction a first-order step, we do it in the prediction step.
\end{remark}

The scheme is finally summarized in the following algorithm.
\begin{algo}~~
\begin{itemize}
\item  Initialization of $(x_k^0, v_k^0)$, $\omega_k^0$ and $\rho_i^0$.

At each time step:

\textbf{Prediction step: from $t^n$ to $t^{n+1/2}$.}
\item 1) Advance micro part: 
\begin{itemize}
\item advance the characteristics with (\ref{eq:cara_pred}), 
\item compute $\langle v g\rangle$ with (\ref{momg}) and B-spline functions of order $\ell\geq 1$, 
\item advance the equation on the weights with (\ref{eq:poids_pred}).
\end{itemize}
\item 2) Projection step: compute $(I-\Pi)g^{n+1/2}$ using \cite{ccl}.
\item 3) Advance macro part: 
\begin{itemize}
\item compute $\langle v g^{n+1/2}\rangle$ with (\ref{momg}) and B-spline functions of order $\ell\geq 1$, 
\item compute the density with (\ref{eq:rho_pred}).
\end{itemize}

\textbf{Correction step: from $t^n$ to $t^{n+1}$.}
\item 4) Advance micro part: 
\begin{itemize}
\item advance the characteristics with (\ref{eq:cara_corr}), 
\item compute $\langle v g\rangle$ with (\ref{momg}) and B-spline functions of order $\ell\geq 1$,
\item advance the equation on the weights with (\ref{eq:poids_corr}).
\end{itemize}
\item 5) Advance macro part with (\ref{eq:rho_corr}).
\end{itemize}
\label{algo:2ndorder}
\end{algo}

\section{Extension to the Vlasov-Poisson-BGK case}\label{ssec:e}
\setcounter{equation}{0}

This section is devoted to the extension of our method to kinetic equation making appear an electric field in the velocity direction. We consider the Vlasov-Poisson-BGK system in the diffusion scaling
\begin{eqnarray}
\partial_t f +\frac{1}{\varepsilon} v\partial_x f  + \frac{1}{\varepsilon}E\partial_v f= \frac{1}{\varepsilon^2}(\rho M - f), \label{eq:vlasovbgk_2} \\
\partial_xE=\rho-1,\label{eq:poisson}\\
\int_\Omega E\dd x=0,~\forall t\geq 0,\label{eq:zeromean}
\end{eqnarray}
where $x\in \Omega=\left[0,L_x\right]\subset\mathbb{R}$, $\rho(t,x)=\int_\mathbb{R} f(t,x,v) \dd v$ and $M\left(v\right)=\frac{1}{\sqrt{2\pi}}\exp\left(-\frac{v^2}{2}\right)$ is the absolute Maxwellian. 
Let $f_0\left(x,v\right)=f\left(t=0,x,v\right)$ the initial distribution function and let consider periodic boundary conditions in $x$: $f\left(t,0,v\right)=f\left(t,L_x,v\right)$, $\forall~v\in V$, $E\left(t,0\right)=E\left(t,L_x\right)$ $\forall t\geq 0$. %The last equation (\ref{eq:zeromean}) is a zero-mean electrostatic condition, necessary to obtain a well-posed problem.

We can extend our schemes to this problem by adapting the computations of Subsections \ref{sec:2ndorder} and \ref{subsec:ordre2_discr}. We do not give all the details of the computations but insist on difficulties coming from the electric field term and write the resulting schemes.

The second-order reformulated micro-macro system corresponding to (\ref{eq:vlasovbgk}) is
\begin{align}
 \partial_t \rho &+ \frac{1}{\varepsilon}\partial_x \langle v g \rangle = 0,\label{eq:mm_macro_2nd_bis}\\
\partial_tg&=\frac{2}{\Delta t}\frac{e^{-\Delta t/\varepsilon^2}-1}{e^{-\Delta t/\varepsilon^2}+1}g-\frac{2}{\varepsilon}\frac{e^{-\Delta t/2\varepsilon^2}}{e^{-\Delta t/\varepsilon^2}+1} \left[ vM\partial_x \rho + v\partial_x g - \partial_x \langle vg\rangle M -vME\rho +E\partial_v g \right].
\label{eq:mm_micro_2nd_bis}
\end{align}
The limit model is here the drift-diffusion equation coupled to Poisson equation (\ref{eq:ddlimit}).

Equipped with this system (\ref{eq:mm_macro_2nd_bis})-(\ref{eq:mm_micro_2nd_bis}), we construct the following Lagrangian, second-order in time scheme for (\ref{eq:vlasovbgk_2})-(\ref{eq:poisson})-(\ref{eq:zeromean}).
In the prediction step, characteristics are solved through
\begin{equation}
x_k^{n+1/2}=x_k^n+\frac{\Delta t}{\varepsilon}\frac{e^{-\Delta t/2\varepsilon^2}}{e^{-\Delta t/\varepsilon^2}+1}v_k^n,~~~v_k^{n+1/2}=v_k^n+\frac{\Delta t}{\varepsilon}\frac{e^{-\Delta t/2\varepsilon^2}}{e^{-\Delta t/\varepsilon^2}+1}E^n(x_k^n),
\label{eq:characwithE}
\end{equation}
the weights evolve with
\begin{equation}
\begin{split}
w_k^{n+1/2}=\frac{2e^{-\Delta t/\varepsilon^2}}{e^{-\Delta t/\varepsilon^2}+1}w_k^n-\frac{\Delta t}{\varepsilon}\frac{e^{-\Delta t/2\varepsilon^2}}{e^{-\Delta t/\varepsilon^2}+1}\left[ v_k^nM(v_k^n)\partial_x \rho^n(x_k^n)- \partial_x \langle v_k^ng^n(x_k^n)\rangle M(v_k^n)\right.  \\
\left.-v_k^nM(v_k^n)E^n(x_k^n)\rho^n(x_k^n)  \right]\frac{L_xL_v}{N_p}
\end{split}
\label{eq:poidswithE}
\end{equation}
and the macro equation is advanced with
\begin{equation}
\rho_i^{n+1/2}=\rho_i^n- \frac{\Delta t}{2\varepsilon}\partial_x \langle v g^{n+1/2} \rangle_i+\frac{\Delta t}{2}(1-e^{-\Delta t/\varepsilon^2})\partial_{x}\left(\partial_x\left(\frac{\rho_i^{n+1/2}+\rho_i^n}{2}\right)-E_i^{n}\rho_i^{n}\right).
\label{eq:rhowithE}
\end{equation}

In the correction step, characteristics are solved through
\begin{equation}
x_k^{n+1}=x_k^n+\frac{2\Delta t}{\varepsilon}\frac{e^{-\Delta t/2\varepsilon^2}}{e^{-\Delta t/\varepsilon^2}+1}v_k^{n+1/2},~~~v_k^{n+1}=v_k^n+\frac{2\Delta t}{\varepsilon}\frac{e^{-\Delta t/2\varepsilon^2}}{e^{-\Delta t/\varepsilon^2}+1}E^{n+1/2}(x_k^{n+1/2}),
\label{eq:characwithE2}
\end{equation}
the weights evolve with
\begin{equation}
\begin{split}
w_k^{n+1}=e^{-\Delta t/\varepsilon^2}w_k^n-\frac{\Delta t}{\varepsilon}e^{-\Delta t/2\varepsilon^2}\left[ v_k^{n+1/2}M(v_k^{n+1/2})\partial_x \rho^{n+1/2}(x_k^{n+1/2})~~~~~~~~~~~~~~~~~~~~~~~~~~~~~~\right.  \\
\left.- \partial_x \langle v_k^{n+1/2}g^{n+1/2}(x_k^{n+1/2})\rangle M(v_k^{n+1/2})~~~~~~~~~~~~~~~~~~~~~~~~~\right.  \\
\left.-v_k^{n+1/2}M(v_k^{n+1/2})E^{n+1/2}(x_k^{n+1/2})\rho^{n+1/2}(x_k^{n+1/2})  \right]\frac{L_xL_v}{N_p}.
\end{split}
\label{eq:poidswithE2}
\end{equation}
and the macro equation is advanced with
\begin{equation}
\rho_i^{n+1}=\rho_i^n- \frac{\Delta t}{\varepsilon}\partial_x \langle v g^{n+1/2} \rangle_i+\Delta t(1-e^{-\Delta t/\varepsilon^2})^2\partial_{x}\left(\partial_x\left(\frac{\rho_i^{n+1}+\rho_i^n}{2}\right)-E_i^{n+1/2}\rho_i^{n+1/2}\right).
\label{eq:vp_lag_2nd_man}
\end{equation}

The limit has been directly written in the macroscopic equation and the diffusion term is managed by a Crank-Nicolson method, in the prediction as well as in the correction step.

We have the following proposition.
\begin{prop}
The scheme given by \eqref{eq:characwithE}-\eqref{eq:poidswithE}-\eqref{eq:rhowithE}-\eqref{eq:characwithE2}-\eqref{eq:poidswithE2}-\eqref{eq:vp_lag_2nd_man} enjoys the AP property, \textit{i.e.} it satisfies the following properties 
\begin{itemize}
\item for fixed $\varepsilon>0$, the scheme is a second-order (in time) approximation of the original model (\ref{eq:vlasovbgk}),% (\ref{eq:mm_macro_2nd_bis})-(\ref{eq:mm_micro_2nd_bis}), 
\item for fixed $\Delta t>0$, the scheme degenerates into an implicit second-order (in time) scheme of \eqref{eq:drift-diff}. 
\end{itemize}
\end{prop}

The scheme is finally summarized in the following algorithm.
\begin{algo}~~
\begin{itemize}
\item  Initialize $(x_k^0, v_k^0)$, $\omega_k^0$,  and $\rho_i^{0}$.
\item Compute $E_i^{0}$ thanks to FFT or finite differences.% (\ref{eq:poissoninit})-(\ref{eq:poissonfd})-(\ref{eq:poissoncorrection}).

At each time step:

\textbf{Prediction step: from $t^n$ to $t^{n+1/2}$.}
\item 1) Advance micro part: 
\begin{itemize}
\item advance the characteristics with (\ref{eq:characwithE}),
\item compute $\langle v g\rangle$ with (\ref{momg}) and B-spline functions of order $\ell\geq 1$, 
\item advance the equation on the weights with (\ref{eq:poidswithE}).
\end{itemize}
%\begin{itemize}
\item 2) Projection step: compute $(I-\Pi)g^{n+1/2}$ using \cite{ccl}.
\item 3) Advance macro part: 
\begin{itemize}
\item compute $\langle v g^{n+1/2}\rangle$ with (\ref{momg}) and B-spline functions of order $\ell\geq 1$, 
\item compute $\rho_i^{n+1/2}$ with (\ref{eq:rhowithE}),
\item compute $E_i^{n+1/2}$ thanks to FFT or finite differences.% (\ref{eq:poissoninit})-(\ref{eq:poissonfd})-(\ref{eq:poissoncorrection}).
\end{itemize}

\textbf{Correction step: from $t^n$ to $t^{n+1}$.}
\item 4) Advance micro part: 
\begin{itemize}
\item advance the characteristics with (\ref{eq:characwithE2}),
\item compute $\langle v g\rangle$ with (\ref{momg}) and B-spline functions of order $\ell\geq 1$,
\item advance the equation on the weights with (\ref{eq:poidswithE2}).
\end{itemize}
%\begin{itemize}
%\item 5) Projection step: compute $(I-\Pi)g^{n+1}$ using \cite{ccl}.
\item 5) Advance macro part: 
\begin{itemize}
\item compute $\rho_i^{n+1}$ with (\ref{eq:vp_lag_2nd_man}),
%\begin{equation}
%\rho_i^{n+1}=\rho_i^n- \frac{\Delta t}{\varepsilon}\partial_x \langle v g^{n+1/2} \rangle_i,
%\label{eq:vp_lag_2nd_sansman}
%\end{equation}
\item compute $E_i^{n+1}$ thanks to FFT or finite differences.% (\ref{eq:poissoninit})-(\ref{eq:poissonfd})-(\ref{eq:poissoncorrection}).
\end{itemize}
\end{itemize}
\label{algo:vp_lag_2nd}
\end{algo}

\begin{remark}
We propose to use an upwind discretization of the derivative $\partial_x\left(E_i^{n+1/2}\rho_i^{n+1/2}\right)$:
$$
\partial_x\left(E_i^{n+1/2}\rho_i^{n+1/2}\right)\approx\frac{E_i^{n+1/2,+}\rho_i^{n+1/2}+E_i^{n+1/2,-}\rho_{i+1}^{n+1/2}-E_{i-1}^{n+1/2,+}\rho_{i-1}^{n+1/2}-E_{i-1}^{n+1/2,-}\rho_i^{n+1/2}}{\Delta x},
$$
where the standard notations $u^+=\mathrm{max}(u,0)$ and $u^-=\mathrm{min}(u,0)$ are used.
The same discretization is done for $\partial_x\left(E_i^{n}\rho_i^{n}\right)$ in the prediction step.
\end{remark}

%\begin{prop}
%The Lagrangian scheme developped for (\ref{eq:vlasovbgk})-(\ref{eq:poisson})-(\ref{eq:zeromean}) and detailed in Algorithm \ref{algo:vp_lag_2nd} is second-order in time and leads to a consistent discretization of the drif-diffusion equation 
%$ \partial_t \rho - \partial_{x}\left(\partial_x\rho-E\rho\right) = 0$.
%\end{prop}

%As in Subsection \ref{subsec:ordre2_discr}, we can modify our discretization to get an AP scheme. The correction we propose consists in computing $\rho_i^{n+1}$ following  
%\begin{equation}
%\rho_i^{n+1}=\rho_i^n- \frac{\Delta t}{\varepsilon}\partial_x \langle v g^{n+1/2} \rangle_i+\Delta t(1-e^{-\Delta t/\varepsilon^2})^2\langle vM\rangle_i^{n+1/2}\partial_{x}\left(\partial_x\rho_i^{n+1/2}-E_i^{n+1/2}\rho_i^{n+1/2}\right).
%\label{eq:vp_lag_2nd_man}
%\end{equation}

%For obtaining a Crank-Nicolson scheme, we can replace $\partial_{x}\rho_i^{n+1/2}$ by $\partial_{x}\left(\frac{\rho_i^{n+1}+\rho_i^n}{2}\right)$ and inverse the resulting system.

\section{Numerical results}\label{sec:numres}
\setcounter{equation}{0}

This section is devoted to some numerical experiments comparing the here designed micro-macro model with particles (denoted by MiMa-Part-1 for the first-order scheme and by MiMa-Part-2 for the second-order scheme) to the micro-macro Eulerian model (denoted by MiMa-Grid), the moment guided method (denoted by Moment G.) and (i) the Full PIC method in kinetic regimes ($\varepsilon$ of order 1) or (ii) the limit scheme in diffusion regime ($\varepsilon\ll 1$). MiMa-Part-1 corresponds to Proposition \ref{prop:lag_1st_imp} and MiMa-Part-2 corresponds to Proposition \ref{prop:lag_2st}. The micro-macro Eulerian scheme is presented in Appendix \ref{subsec:eulerian}. The moment guided method was first presented in \cite{ddp} and is adapted to our context in Appendix \ref{app:mgm}. The Full PIC method \cite{birdsall} consists in applying the particle representation (\ref{eq:diracmasses}) to the whole function $f$ (and not only to the perturbation $g$) and to solve the characteristics and the equations on weights coming from equation \eqref{eq:etrbgk} or \eqref{eq:vlasovbgk}. %This method is known to be affected by numerical noise due to its probabilistic character and the nonlinearity. This noise diminishes when the number of considered particles $N_p$ increases, but its cost raises too. The micro-macro Lagrangian schemes we propose are expected to be less noisy (and so less costly), since the noise affects only the perturbation $g$. Moreover, our micro-macro schemes (Eulerian and Lagrangian), as well as the moment guided one, are AP.

In the sequel, we consider three families of test cases: radiative transport equation (RTE) test cases with periodic boundary conditions (see equation (\ref{eq:etrbgk})) in Subsection \ref{ssec:RTEperio}, % and with Dirichlet boundary conditions in Subsection \ref{RTEdirichlet}
Vlasov-Poisson-BGK test cases (see equations (\ref{eq:vlasovbgk})-(\ref{eq:vlasovbgkpoisson})) of Landau damping type in Subsection \ref{ssec:landau} and two-stream instability (TSI) test cases in Subsection \ref{ssec:tsi}.

%\subsection{RTE test cases}

\subsection{RTE with periodic boundary conditions}\label{ssec:RTEperio}

We consider the RTE test case given by the initial condition
\begin{equation}
f\left(t=0,x,v\right)=1+\cos\left(2\pi\left(x+\frac{1}{2}\right)\right),~~~x\in\left[0,1\right], v\in\left[-1,1\right],
\end{equation}
with $M\left(v\right)=1,~\forall v\in\left[-1,1\right]$, and periodic boundary conditions in $x$.

%Preliminary tests have been presented in \cite{ccl2} for both micro-macro first-order in time schemes (Eulerian and Lagrangian in space). The density $\rho$ was shown to be well computed with the Eulerian as well as the Lagrangian scheme. The micro-macro Lagrangian scheme was compared to a full PIC method: with the same number of particles, the micro-macro scheme was less noisy than the full PIC one. In other words, for a given accuracy of the solution, the micro-macro scheme was less costly than the full PIC one.

We propose here to verify numerically the convergence of MiMa-Part-2 (presented in Subsection \ref{ssec:pic_o2}), Lagrangian in space, of second-order in time and with an implicit treatment of the diffusion term (see Algorithm \ref{algo:2ndorder}). In Figure \ref{fig:etr_ordre2_1}, we plot the error in $L^\infty$ norm of the density $\rho$ at time $T=0.1$ as a function of $\Delta t$ (from $10^{-4}$ to $0.1$) for the following parameters: $N_x=16$, $N_p=100$. For $\varepsilon=1$, $0.5$ and $0.1$, the reference solution is computed with MiMa-Part-2 using the same parameters but with $\Delta t=10^{-7}$. Whereas for $\varepsilon=10^{-6}$, the reference is a numerical solution of the diffusion equation (computed on a space grid, with $N_x=16$ and $\Delta t=10^{-7}$).
In Figure \ref{fig:etr_ordre2_cloche}, the error in $L^\infty$ norm is now represented as a function of $\varepsilon$ for different values of $\Delta t$: $10^{-1}$, $10^{-2}$, $10^{-3}$ and $10^{-4}$.

\begin{minipage}[t]{0.48\textwidth}
\begin{figure}[H]
\includegraphics[angle=-90,width=\textwidth]{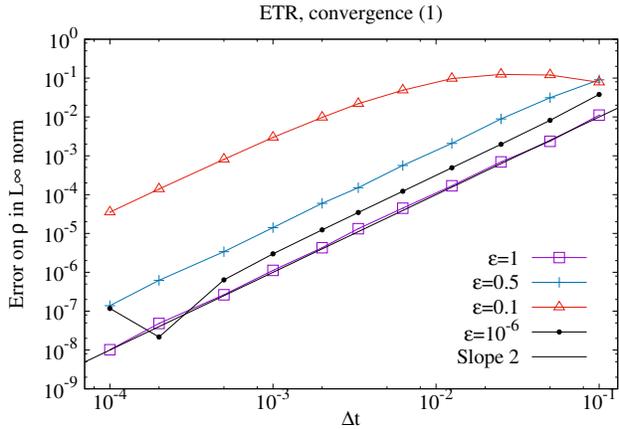}%~~~~~\includegraphics[angle=-90,width=0.48\textwidth]{Figures_article/erreurs_T1e-1_nx16_np100_impl_limite.eps}
\caption{Error in $L^\infty$ norm of $\rho$ at time $T=0.1$ as a function of $\Delta t$ for $N_x=16$, $N_p=100$, $\varepsilon=1$, $0.5$, $0.1$ and $10^{-6}$.}
\label{fig:etr_ordre2_1}
\end{figure}
\end{minipage}
~~
\begin{minipage}[t]{0.48\textwidth}
\begin{figure}[H]
\includegraphics[angle=-90,width=\textwidth]{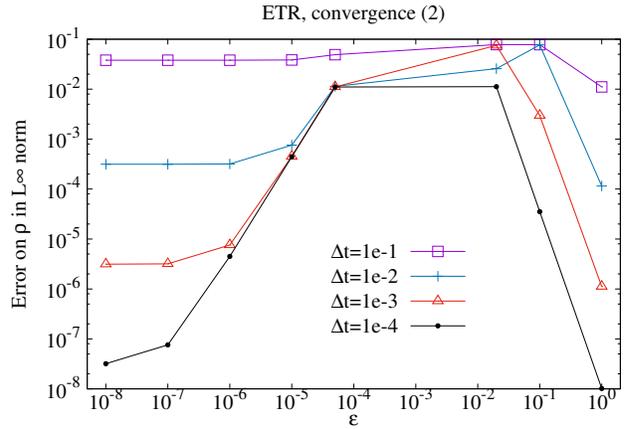}%~~~~~\includegraphics[angle=-90,width=0.48\textwidth]{Figures_article/erreurs_T1e-1_nx16_np100_impl_limite.eps}
\caption{Error in $L^\infty$ norm of $\rho$ at time $T=0.1$ as a function of $\varepsilon$ for $N_x=16$, $N_p=100$, $\Delta t=10^{-1}$, $10^{-2}$, $10^{-3}$ and $10^{-4}$.}
\label{fig:etr_ordre2_cloche}
\end{figure}
\end{minipage}
~~

These plots confirm the fact that MiMa-Part-2 is second-order accurate in time for any fixed $\varepsilon>0$, and also when $\varepsilon\to 0$. However, for intermediate regimes (for instance $\varepsilon=0.1$), order reduction is observed. This is a classical observation for AP schemes. %Schemes conserving the order even when $\varepsilon=\mathcal{O}(\Delta t)$ are called asymptotically accurate, but we did not aim to have this property in this paper. 
Note that similar behaviour is obtained with $L^2$ norm. 

In Figure \ref{fig:etr_ordre2_ap}, we verify the AP property of the MiMa-Part-2 scheme and plot the density $\rho(T=0.1,x)$ as a function of $x$ for different values of $\varepsilon$: $1$, $0.25$, $10^{-2}$ and $10^{-6}$. We take fixed parameters: $N_x=64$, $\Delta t=10^{-3}$ and $N_p=10^4$. We compare the solutions obtained by MiMa-Part-2 to a numerical solution of the diffusion equation (computed on a space grid, with $N_x=512$ and $\Delta t=\Delta x^2$) and see that the $\varepsilon$-dependent solutions come closer to the diffusion one when $\varepsilon$ decreases.

Moreover, we illustrate in Figure \ref{fig:etr_ordre2_npart} the fact that the cost of our method is very small at the limit. For that, we plot the density $\rho(T=0.1,x)$ as a function of $x$ for $\varepsilon=10^{-6}$ ($N_x=64$ and $\Delta t=10^{-2}$) and see that $N_p=100$ is sufficient to represent in a good way (without noise) the density. The numerical cost is then very close to the one of the asymptotic model.

% The noise coming from the particle discretization is very small when $\varepsilon=5\times10^{-3}$, so $N_p=100$ is sufficient, whereas we need to take $N_p=2\times10^{4}$ for $\varepsilon=10^{-1}$ and $N_p=5\times10^{4}$ for $\varepsilon=2\times10^{-1}$ to obtain smooth densities. Concerning the time step $\Delta t$, we take $\Delta t=10^{-2}$ for $\varepsilon=5\times10^{-3}$ and $\varepsilon=2\times10^{-1}$ and $\Delta t=10^{-3}$ for $\varepsilon=10^{-1}$. We compare these solutions to a numerical solution of the diffusion equation (computed on a phase-space grid, with $N_x=512$ and $\Delta t=\Delta x^2$) and see that the $\varepsilon$-dependent solutions come closer to the diffusion one when $\varepsilon$ reduces, and that the limit is reached for $\varepsilon=5\times10^{-3}$ (let us add that $\varepsilon<5\times10^{-3}$ gives the same solution as for $\varepsilon=5\times10^{-3}$).

\begin{minipage}[t]{0.48\textwidth}
\begin{figure}[H]
\includegraphics[angle=-90,width=\textwidth]{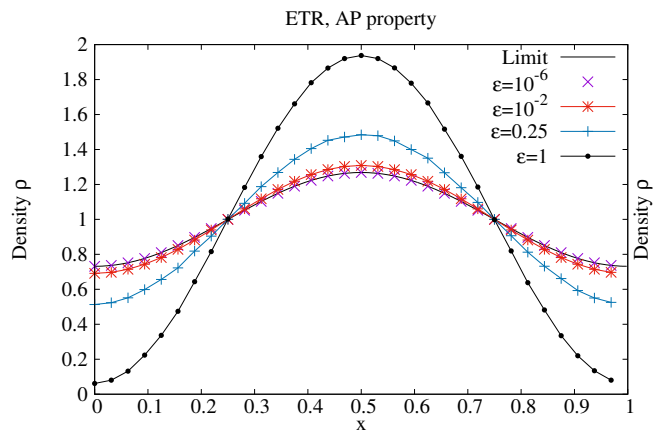}%~~~~~\includegraphics[angle=-90,width=0.48\textwidth]{Figures_article/erreurs_T1e-1_nx16_np100_impl_limite.eps}
\caption{AP property. Density $\rho(T=0.1,x)$ for $\varepsilon=1$, $0.25$, $10^{-2}$ and $10^{-6}$. $N_x=64$, $\Delta t=10^{-3}$ and $N_p=10^4$. Comparison with the diffusion solution.}
\label{fig:etr_ordre2_ap}
\end{figure}
\end{minipage}
~~
\begin{minipage}[t]{0.48\textwidth}
\begin{figure}[H]
\includegraphics[angle=-90,width=\textwidth]{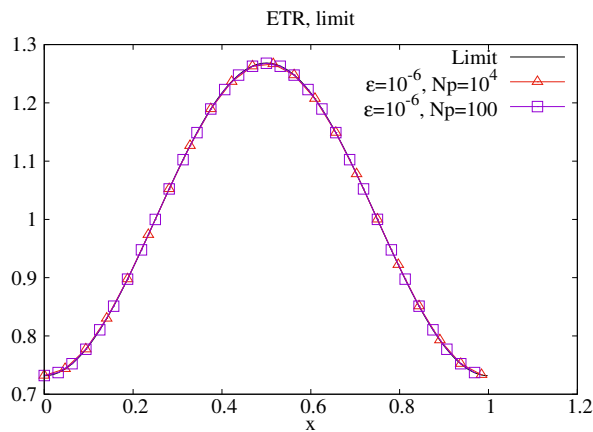}%~~~~~\includegraphics[angle=-90,width=0.48\textwidth]{Figures_article/erreurs_T1e-1_nx16_np100_impl_limite.eps}
\caption{Cost at the limit. Density $\rho(T=0.1,x)$ for $\varepsilon=10^{-6}$, $N_x=64$, $\Delta t=10^{-2}$ and $N_p=10^4$ or $N_p=100$. Comparison with the diffusion solution.}
\label{fig:etr_ordre2_npart}
\end{figure}
\end{minipage}
~~

\subsection{Landau damping}\label{ssec:landau}

In this subsection, we present Landau damping test cases in both regimes (kinetic when $\varepsilon=\mathcal{O}(1)$ and diffusive when $\varepsilon\to 0$). The initial distribution function is given by 
\begin{equation}
f\left(t=0,x,v\right)=\frac{1}{\sqrt{2\pi}}\textrm{exp}\left(-\frac{v^2}{2}\right)\left(1+\alpha\cos\left(kx\right)\right),~~~x\in\left[0,\frac{2\pi}{k}\right],~v\in \mathbb{R},
\end{equation}
with the wave number $k=0.5$ and $\alpha=0.05$. For the micro-macro model (\ref{eq:vlasovbgk}), the initial condition is $\rho\left(t=0,x\right)=1+\alpha\cos\left(kx\right)$ and $g\left(t=0,x,v\right)=0$. For the limit drift-diffusion equation (\ref{eq:drift-diff}), we have $\rho\left(t=0,x\right)=1+\alpha\cos\left(kx\right)$.

We first verify the order in time of the MiMa-Part-2 scheme detailed in Section \ref{ssec:e} and plot in Figure \ref{fig:vp_ordre2_1} the error in $L^\infty$ norm of the density $\rho$ at time $T=0.1$ as a function of $\Delta t$ (from $10^{-4}$ to $0.1$) for the following parameters: $N_x=16$, $N_p=100$. For $\varepsilon=1$, $0.5$ and $0.1$, the reference solution is computed with MiMa-Part-2 using the same parameters but with $\Delta t=10^{-7}$. Whereas for $\varepsilon=10^{-6}$ the reference is a numerical solution of the drift-diffusion equation (computed on a space grid, with $N_x=16$ and $\Delta t=10^{-7}$).

Results are similar to the RTE case: the second-order in time is preserved for big and small values of $\varepsilon$ but not for intermediate regimes.% when $\varepsilon$ is of order $\Delta t$. 

We are now interested in more qualitative tests by considering the time history of the electric energy $\mathcal{E}\left(t\right)=\sqrt{\int_0^{L_x} E\left(t,x\right)^2\dd x}$ in semi-logarithmic scale for different values of $\varepsilon$. We compare the results obtained by MiMa-Part-2 (detailed in Algorithm \ref{algo:vp_lag_2nd}) to other schemes: MiMa-part-1, Moment G. and Full PIC (for $\varepsilon$ of order 1) or the scheme for the drift-diffusion model (for small values of $\varepsilon$).
 
We expect that the number of particles that is necessary to represent in a good way the perturbation $g$ in MiMa-Part methods decreases when $\varepsilon$ diminishes and consider $N_p=10^{5}$ if $\varepsilon\geq 0.5$, $N_p=10^{4}$ if $\varepsilon=0.1$ and $N_p=10^{2}$ if $\varepsilon=10^{-4}$. For comparison, we take the same $N_p$ for moment guided and Full PIC methods.

%For the MiMa-Grid and the Limit methods, we take $N_x=N_v=64$ if $\varepsilon<1$ and $N_x=N_v=512$ if $\varepsilon\geq 1$.% for MiMa-Grid, and $N_x=256$, $N_v=128$ for the Limit scheme. 
%
%For MiMa-Part and Moment G. methods, we take $\Delta t=\frac{\Delta x^2}{2}$ if $\varepsilon<1$ and $\Delta t=10^{-2}\varepsilon$ if $\varepsilon\geq1$. For full PIC, we take $\Delta t=\max\left(10^{-2},10^{-2}\varepsilon\right)$. For MiMa-Grid, we take $\Delta t=\frac{1}{10}\min\left(\frac{\Delta x}{\max(v)},\Delta x^2\right)$ and for the Limit scheme, we take $\Delta t=10^{-3}$.

Results for $\varepsilon=10$ are given in Figure \ref{figLandaukin10}. For the four particle methods, we take $\Delta t=0.1$, and $N_x=128$. With the same parameters, results for $\varepsilon=1$ are presented in Figure \ref{figLandaukin}. 
For $\varepsilon=0.5$, we consider $\Delta t=0.01$ and $N_x=256$ for the four particle methods. Results are given in Figure \ref{figLandauinter}. For these three values of $\varepsilon$, the reference is given by MiMa-Grid with $N_x=N_v=512$ and $\Delta t=\Delta x^2\approx 6\times 10^{-4}$. First, we note that the behaviour of the electric energy is well described during time by micro-macro schemes (MiMa-Part and MiMa-Grid). As observed in \cite{ccl}, the Full PIC method suffers from numerical noise. This is due to the probabilistic character of particle methods (for instance the random initialization of particles). To reduce this noise, we should consider more particles, which would increase the numerical cost. As expected, the moment guided method gives better results than the Full PIC one, but suffers however also from this noise. In MiMa-Part schemes, only the perturbation $g$ is represented by particles (not the whole distribution function $f$), that is why for the same $N_p$, the noise is lower, which enables to capture the reference solution for large times. In addition, MiMa-Part-2 is closer to the reference MiMa-Grid than the first-order version MiMa-Part-1.

%small $\varepsilon$ : comparison between mima-grid, mima-part, moment guided and limit scheme`

\begin{minipage}[t]{0.48\textwidth}
\begin{figure}[H]
\includegraphics[angle=-90,width=\textwidth]{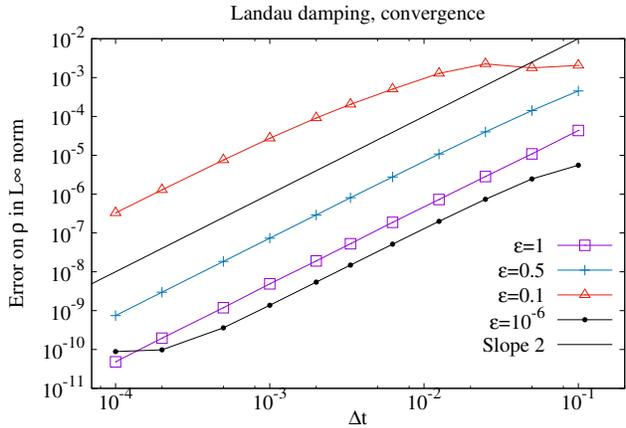}%~~~~~\includegraphics[angle=-90,width=0.48\textwidth]{Figures_article/erreurs_T1e-1_nx16_np100_impl_limit_vp.eps}
\caption{Error in $L^\infty$ norm of $\rho$ at time $T=0.1$ as a function of $\Delta t$ for $N_x=16$, $N_p=100$, $\varepsilon=1$, $0.5$, $0.1$ and $10^{-6}$.}% for $N_x=16$, $N_p=100$, $\varepsilon=1$, $0.5$, $0.1$ and $10^{-6}$.}
\label{fig:vp_ordre2_1}
\end{figure}
\end{minipage}
~~
\begin{minipage}[t]{0.48\textwidth}
\begin{figure} [H]
\includegraphics[angle=-90,width=\textwidth]{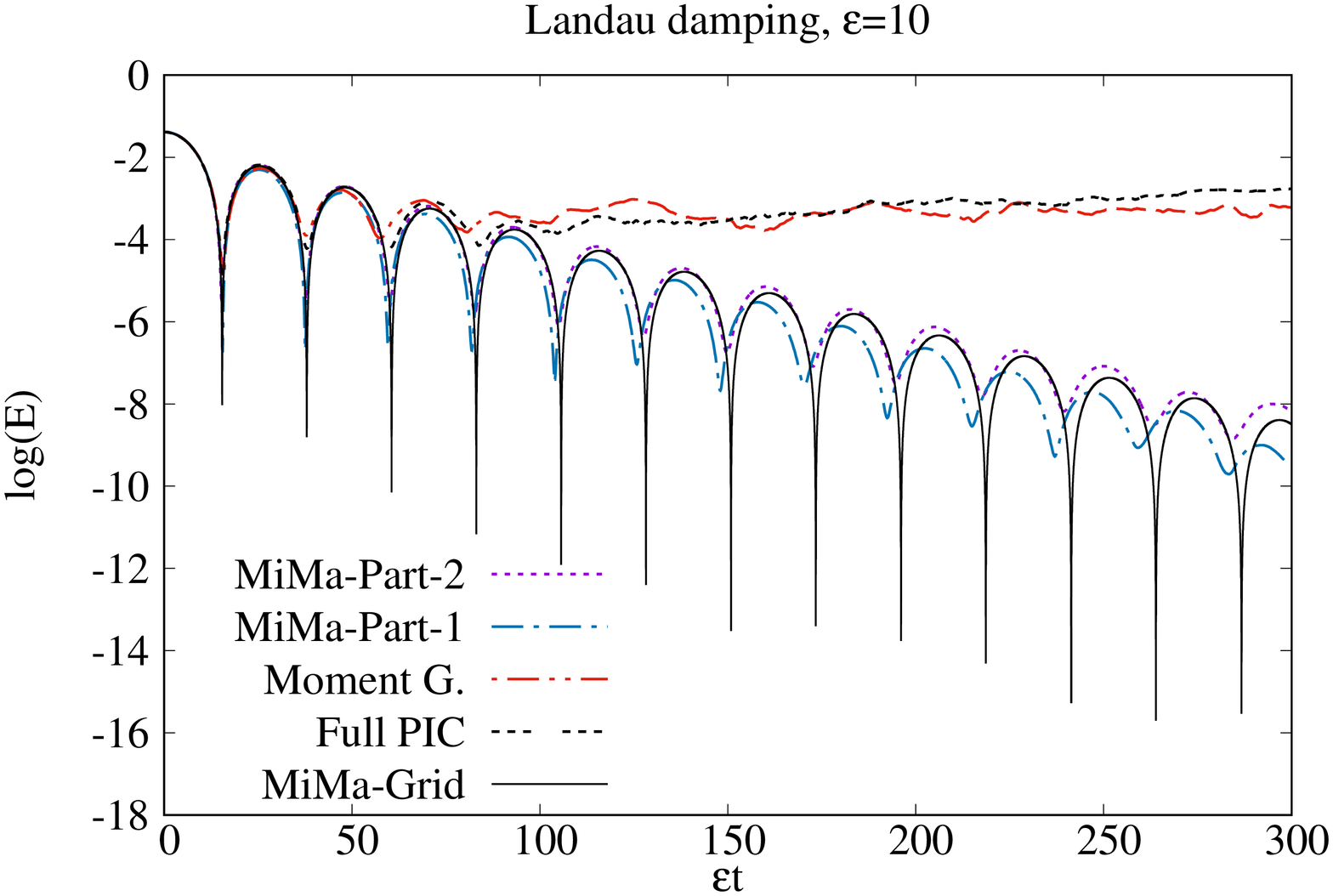}
\caption{Time history of the electric energy, $\varepsilon=10$. $\Delta t=0.1$, $N_x=128$ and $N_p=10^5$ for the four particle methods.}% $\Delta t=\Delta x^2$ and $N_x=N_v=512$ for MiMa-Grid.}
\label{figLandaukin10} 
\end{figure}
\end{minipage}

\begin{minipage}[t]{0.48\textwidth}
\begin{figure} [H]
\includegraphics[angle=-90,width=\textwidth]{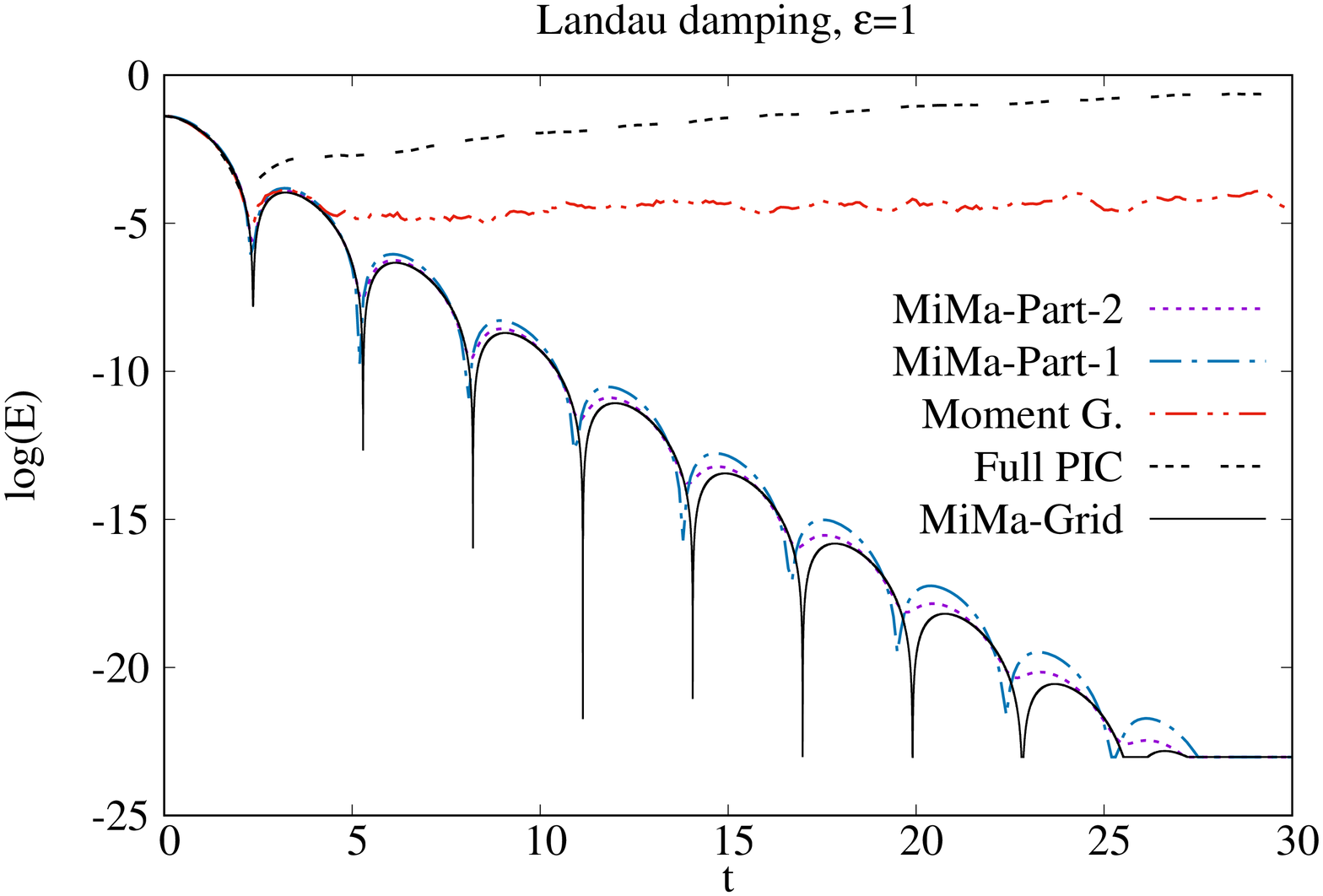}
\caption{Time history of the electric energy, $\varepsilon=1$. $\Delta t=0.1$, $N_x=128$ and $N_p=10^5$ for the four particle methods.}% $\Delta t=\Delta x^2$ and $N_x=N_v=512$ for MiMa-Grid.}
\label{figLandaukin} 
\end{figure}
\end{minipage}
~~
\begin{minipage}[t]{0.48\textwidth}
\begin{figure} [H]
\includegraphics[angle=-90,width=\textwidth]{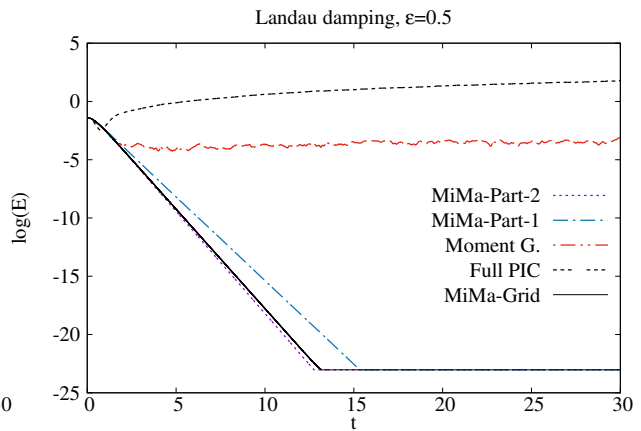}
\caption{Time history of the electric energy, $\varepsilon=0.5$. $\Delta t=0.01$, $N_x=256$ and $N_p=10^5$ for the four particle methods.}% $\Delta t=\Delta x^2$ and $N_x=N_v=512$ for MiMa-Grid.}
\label{figLandauinter} 
\end{figure}
\end{minipage}
~~

For smaller values of $\varepsilon$, we compare the four AP schemes (MiMa-Part-1, MiMa-Part-2, MiMa-Grid and Moment G.) to the limit scheme. Results for $\varepsilon=0.1$ are given in Figure \ref{figLandaudif}. Parameters are the following: $\Delta t=10^{-3}$ and $N_x=128$ for particle methods, $\Delta t=0.1\Delta x^2\approx 3.5\times 10^{-3}$ and $N_x=N_v=64$ for MiMa-Grid. We observe that MiMa-Part-2 is the best method since it almost coincides with the reference MiMa-Grid method. Finally, results for $\varepsilon=10^{-4}$ are given in Figure \ref{figLandaudiflim}, where we have $\Delta t=10^{-2}$ and $N_x=128$ for particle methods, $\Delta t=0.1\Delta x^2\approx 3.5\times 10^{-3}$ and $N_x=N_v=64$ for MiMa-Grid. The asymptotic regime is well recovered by all these AP methods.

\begin{minipage}[t]{0.48\textwidth}
\begin{figure} [H]
\includegraphics[angle=-90,width=\textwidth]{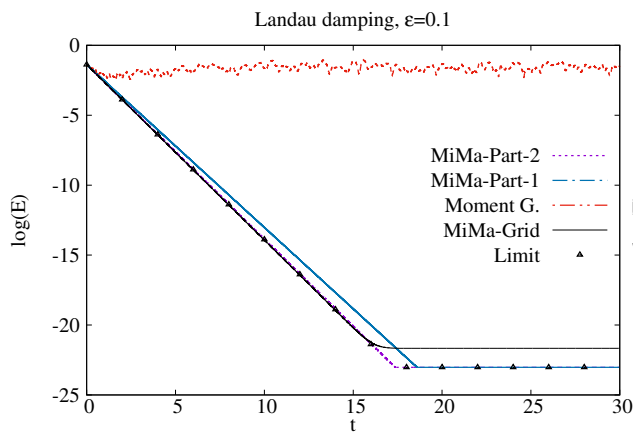}~~~~
\caption{Time history of the electric energy, $\varepsilon=0.1$. $\Delta t=10^{-3}$, $N_x=128$ and $N_p=10^4$ for the four particle methods.}% $\Delta t=0.1\Delta x^2$ and $N_x=N_v=64$ for MiMa-Grid.}
\label{figLandaudif} 
\end{figure}
\end{minipage}
~~
\begin{minipage}[t]{0.48\textwidth}
\begin{figure} [H]
\includegraphics[angle=-90,width=\textwidth]{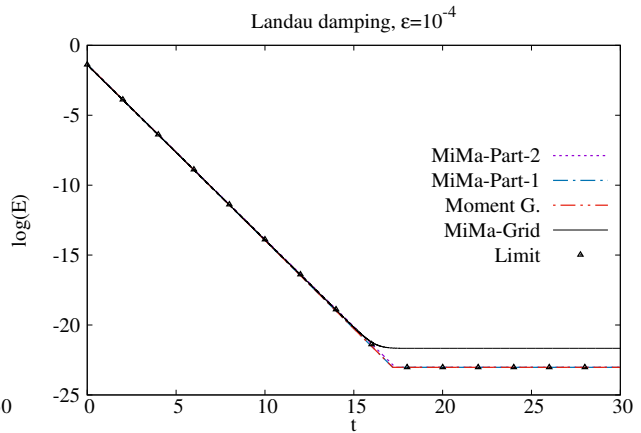}
\caption{Time history of the electric energy, $\varepsilon=10^{-4}$. $\Delta t=0.01$, $N_x=128$ and $N_p=100$ for the four particle methods.}% $\Delta t=0.1\Delta x^2$ and $N_x=N_v=64$ for MiMa-Grid.}
\label{figLandaudiflim} 
\end{figure}
\end{minipage}
~~

As in \cite{ccl}, we remark that few particles are sufficient in the particle-micro-macro schemes to describe in a good way the solution when $\varepsilon$ is small. The cost  is then reduced at the limit.

%intermediate $\varepsilon$ : comparison between mima-grid, mima-part, moment guided and full-pic

%large $\varepsilon$ : comparison between mima-grid, mima-part, moment guided and full-pic

\subsection{Two stream instability}\label{ssec:tsi}

We propose now a study in which the perturbation $g$ is not zero initially and consider the Two-Stream Instability (TSI) test case in both regimes (kinetic and diffusive). The initial distribution function is given by 
\begin{equation}
f\left(t=0,x,v\right)=\frac{1}{\sqrt{2\pi}}v^2\textrm{exp}\left(-\frac{v^2}{2}\right)\left(1+\alpha\cos\left(kx\right)\right),~~~x\in\left[0,\frac{2\pi}{k}\right],~v\in\mathbb{R},
\end{equation}
with the wave number $k=0.5$ and $\alpha=0.05$. The initial condition for the micro-macro model (\ref{eq:vlasovbgk}) is $\rho\left(t=0,x\right)=1+\alpha\cos\left(kx\right)$ and $g\left(t=0,x,v\right)=\frac{1}{\sqrt{2\pi}}\left(v^2-1\right)\textrm{exp}\left(-\frac{v^2}{2}\right)\left(1+\alpha\cos\left(kx\right)\right)$. For the limit drift-diffusion equation (\ref{eq:drift-diff}), we have as in the Landau damping case $\rho\left(t=0,x\right)=1+\alpha\cos\left(kx\right)$.

We first verify the order in time of the MiMa-Part-2 scheme detailed in Section \ref{ssec:e} and plot in Figure \ref{fig:tsi_ordre2_1} the error in $L^\infty$ norm of the density $\rho$ at time $T=0.1$ as a function of $\Delta t$ (from $10^{-4}$ to $0.1$) for the following parameters: $N_x=16$, $N_p=100$. For $\varepsilon=1$, $0.5$ and $0.1$, the reference solution is computed with MiMa-Part-2 using the same parameters but with $\Delta t=10^{-7}$. Whereas for $\varepsilon=10^{-6}$, the reference is a numerical solution of the drift-diffusion equation (computed on a space grid, with $N_x=16$ and $\Delta t=10^{-7}$).

As for the RTE and the Landau damping cases, the second-order in time is preserved for big and small values of $\varepsilon$ but not for intermediate regimes.% when $\varepsilon$ is of order $\Delta t$.

We are now interested in the time evolution of the electric energy $\mathcal{E}\left(t\right)=\sqrt{\int_0^{L_x} E\left(t,x\right)^2\dd x}$ in all regimes. %and take the parameters as for the Landau damping test case (see Subsection \ref{ssec:landau}). 

Results for $\varepsilon=10$ are given in Figure \ref{figTSIkin10}. For the four particle methods, we take $N_p=10^{6}$, $\Delta t=0.1$, and $N_x=128$. By taking $N_p=10^{5}$, $\Delta t=0.1$, and $N_x=128$ for particle methods, we obtain results for $\varepsilon=1$ presented in Figure \ref{figTSIkin}. For $\varepsilon=0.5$, we consider $N_p=10^{5}$, $\Delta t=0.01$ and $N_x=256$ for the four particle methods. Results are given in Figure \ref{figTSIinter}. For these three values of $\varepsilon$, the reference is given by MiMa-Grid with $N_x=N_v=512$ and $\Delta t=\Delta x^2\approx 6\times 10^{-4}$. The behaviour of the electric energy is well described during time by micro-macro schemes (MiMa-Part and MiMa-Grid). As previously, the Full PIC method, as well as moment guided method suffer from numerical noise.

\begin{minipage}[t]{0.48\textwidth}
\begin{figure}[H]
\includegraphics[angle=-90,width=\textwidth]{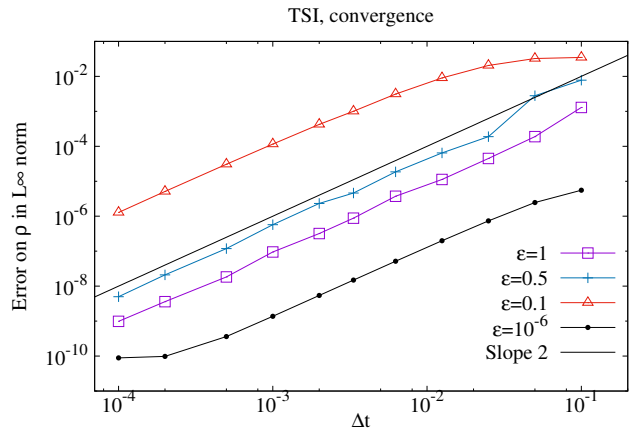}%~~~~~\includegraphics[angle=-90,width=0.48\textwidth]{Figures_article/erreurs_T1e-1_nx16_np100_impl_limit_vp.eps}
\caption{Error in $L^\infty$ norm of $\rho$ at time $T=0.1$ as a function of $\Delta t$  for $N_x=16$, $N_p=100$, $\varepsilon=1$, $0.5$, $0.1$ and $10^{-6}$.}% for $N_x=16$, $N_p=100$, $\varepsilon=1$, $0.5$, $0.1$ and $10^{-6}$.}
\label{fig:tsi_ordre2_1}
\end{figure}
\end{minipage}
~~
\begin{minipage}[t]{0.48\textwidth}
\begin{figure} [H]
\includegraphics[angle=-90,width=\textwidth]{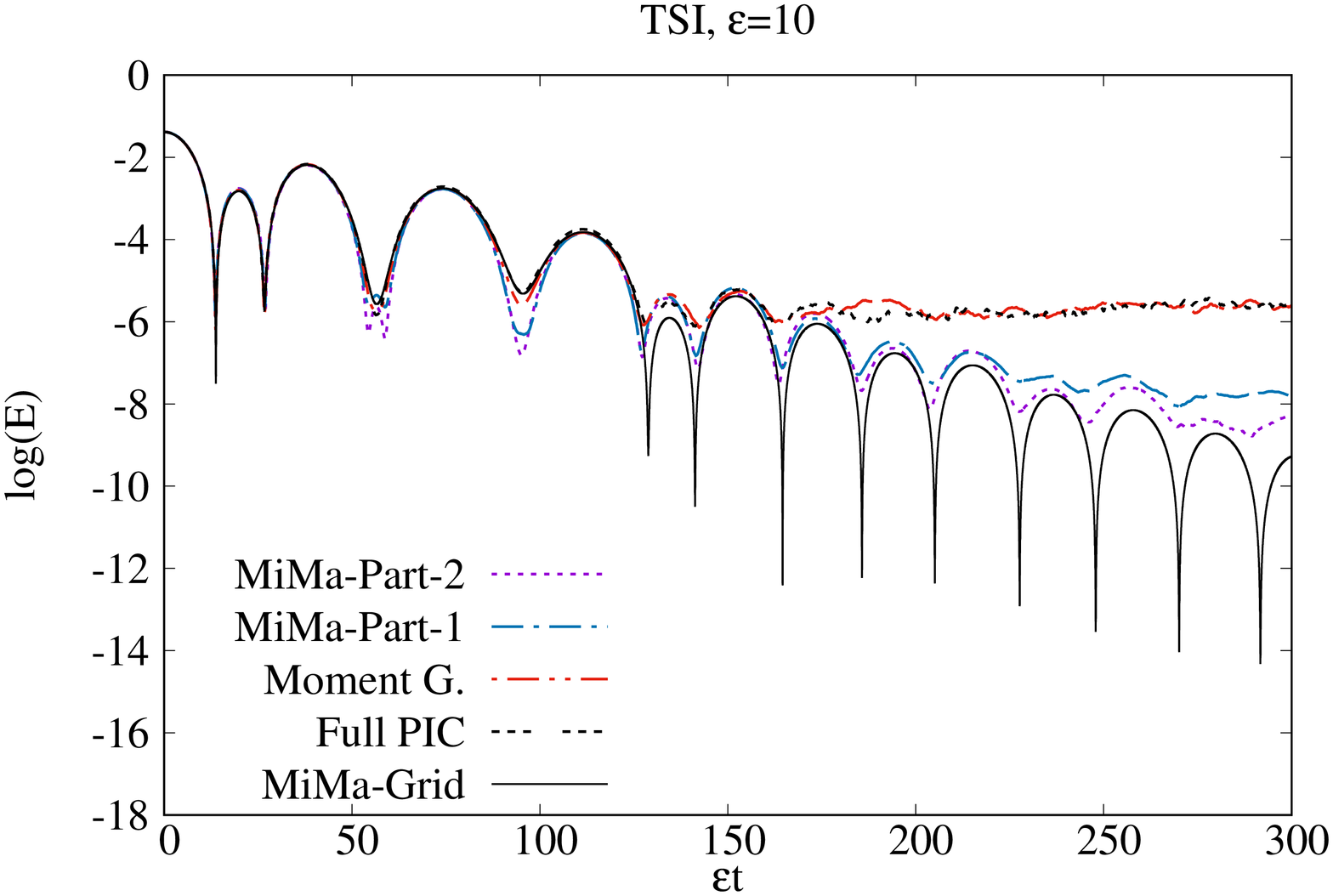}
\caption{Time history of the electric energy, $\varepsilon=10$. $\Delta t=0.1$, $N_x=128$ and $N_p=10^6$ for the four particle methods.}% $\Delta t=\Delta x^2$ and $N_x=N_v=512$ for MiMa-Grid.}
\label{figTSIkin10} 
\end{figure}
\end{minipage}

\begin{minipage}[t]{0.48\textwidth}
\begin{figure} [H]
\includegraphics[angle=-90,width=\textwidth]{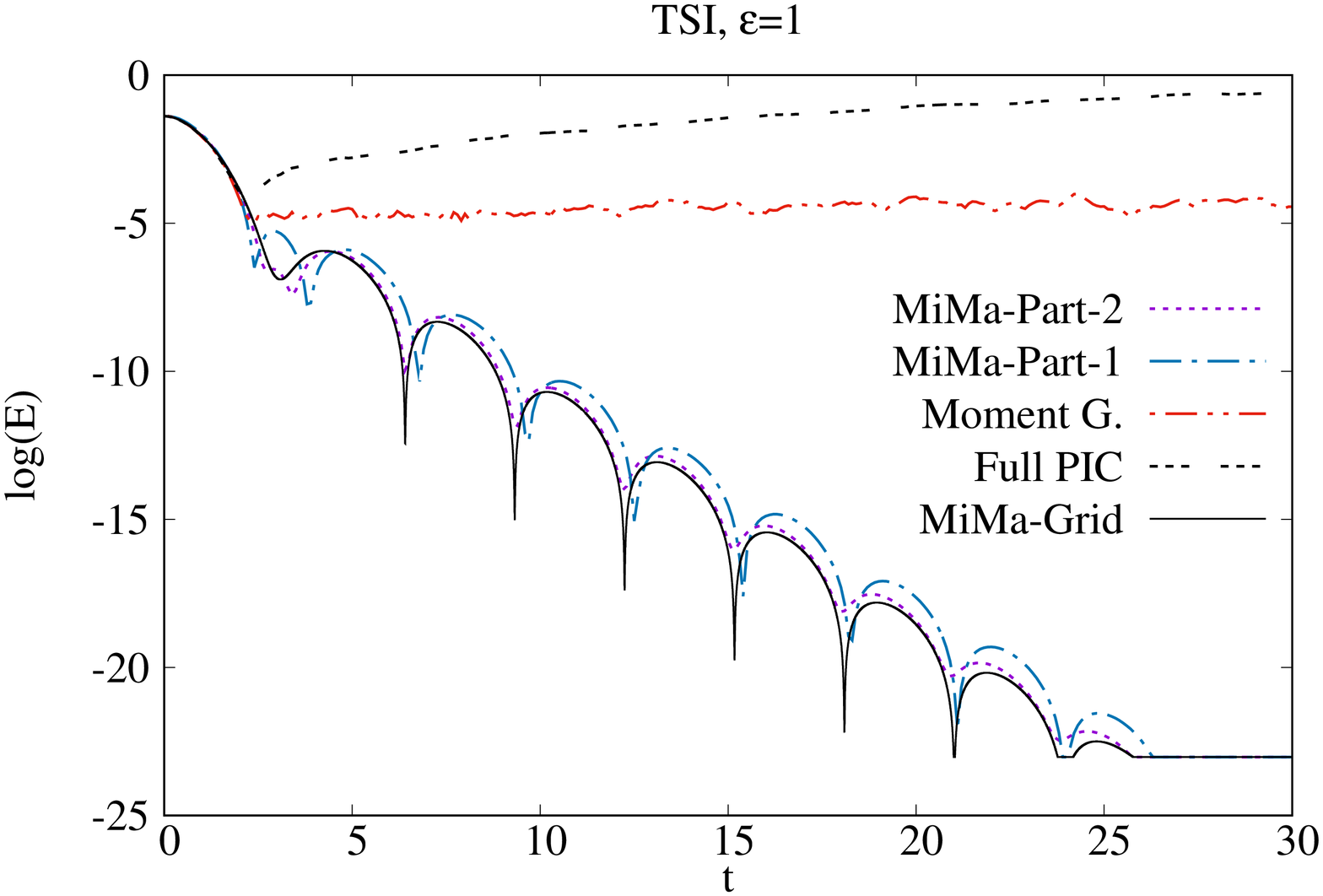}
\caption{Time history of the electric energy, $\varepsilon=1$. $\Delta t=0.1$, $N_x=128$ and $N_p=10^5$ for the four particle methods.}% $\Delta t=\Delta x^2$ and $N_x=N_v=512$ for MiMa-Grid.}
\label{figTSIkin} 
\end{figure}
\end{minipage}
~~
\begin{minipage}[t]{0.48\textwidth}
\begin{figure} [H]
\includegraphics[angle=-90,width=\textwidth]{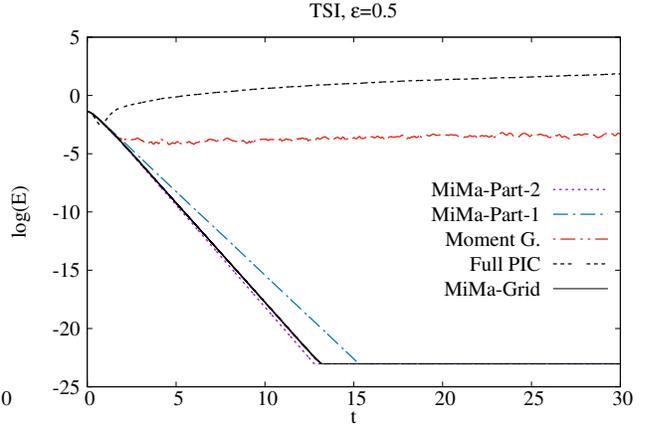}
\caption{Time history of the electric energy, $\varepsilon=0.5$. $\Delta t=0.01$, $N_x=256$ and $N_p=10^5$ for the four particle methods.}% $\Delta t=\Delta x^2$ and $N_x=N_v=512$ for MiMa-Grid.}
\label{figTSIinter} 
\end{figure}
\end{minipage}
~~

To illustrate the efficiency of the method, we plot $f(T=5,x,v)$ obtained by the reference MiMa-Grid, by MiMa-Part-2 and by Full PIC for $\varepsilon=10$ on Figure \ref{fig3Deps10} and for $\varepsilon=0.5$ on Figure \ref{fig3Deps05}. For MiMa-Grid and MiMa-Part-2, $f$ is reconstruted from $g$, $\rho$ and $M$, whereas the approximation of $f$ is directly given by the Full PIC scheme. The numerical parameters are the same as previously (see comments on Figures \ref{figTSIkin10} and \ref{figTSIinter}). On Figure \ref{fig3Deps10}, we observe that the result obtained by MiMa-Part-2 and Full PIC are in good agreement with MiMa-Grid; however, some numerical noise can be distinguished on $f$ obtained by the Full PIC method. On Figure \ref{fig3Deps05} ($\varepsilon=0.5$), we can see clearly that the level of the noise is higher for Full PIC, which prevents it from giving good results. On the contrary, MiMa-Part-2 produces good results compared to MiMa-Grid, since the noise only affects the micro part $g$, which is small in this regime.

\begin{figure} [H]
\includegraphics[angle=-90,width=0.4\textwidth]{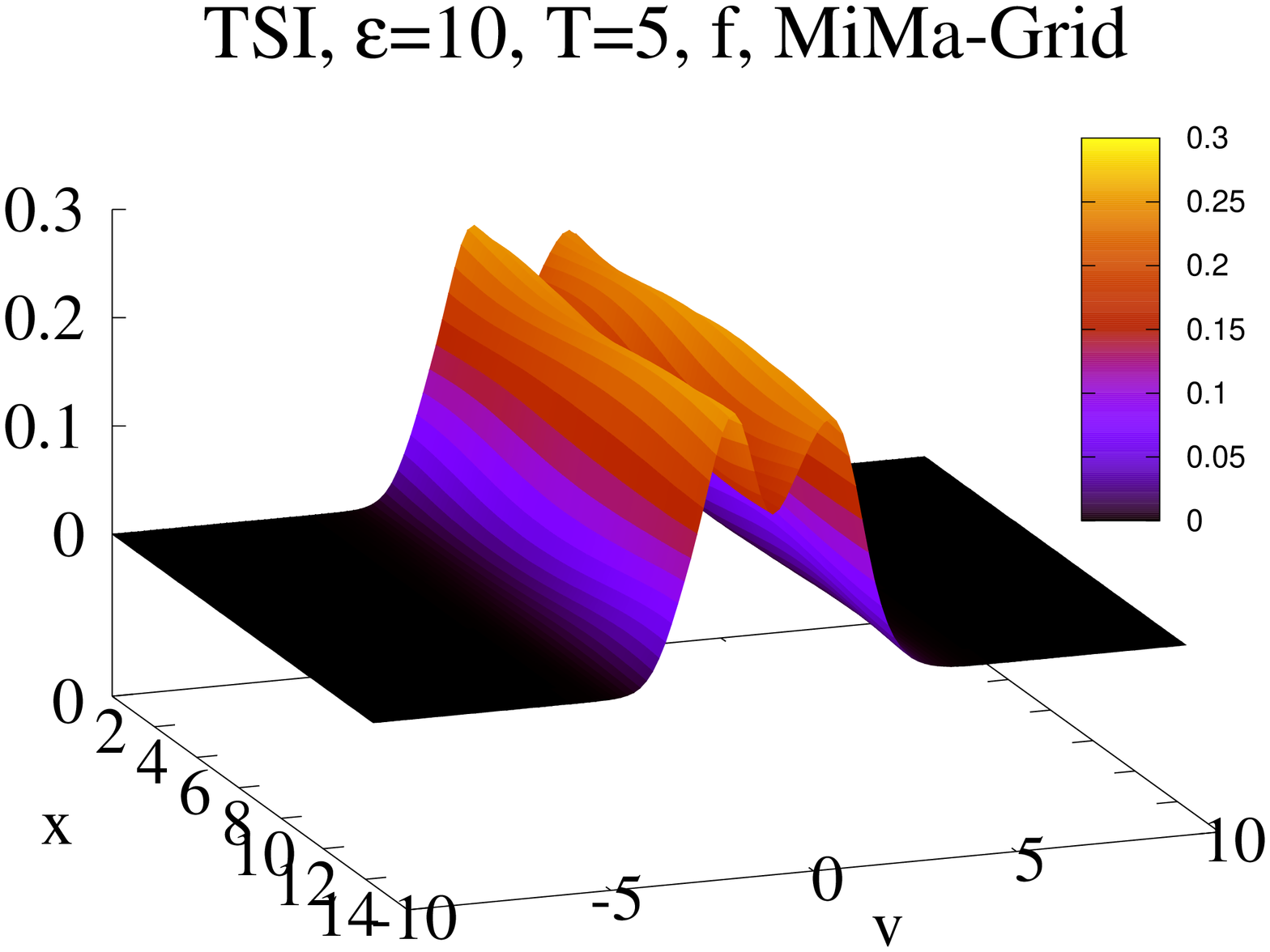}~~\hspace{-1.6cm}~~
\includegraphics[angle=-90,width=0.4\textwidth]{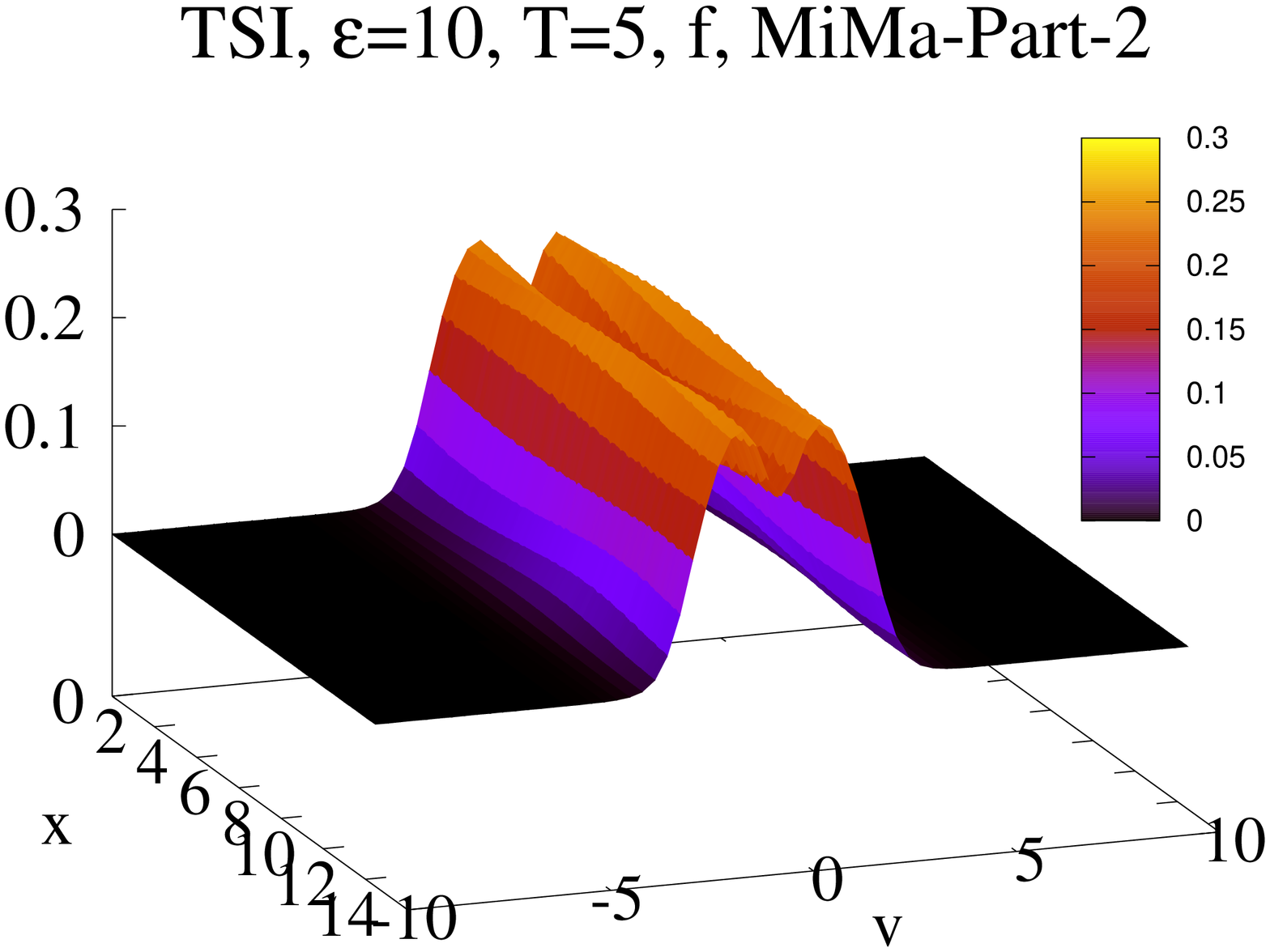}~~\hspace{-1.6cm}~~
\includegraphics[angle=-90,width=0.4\textwidth]{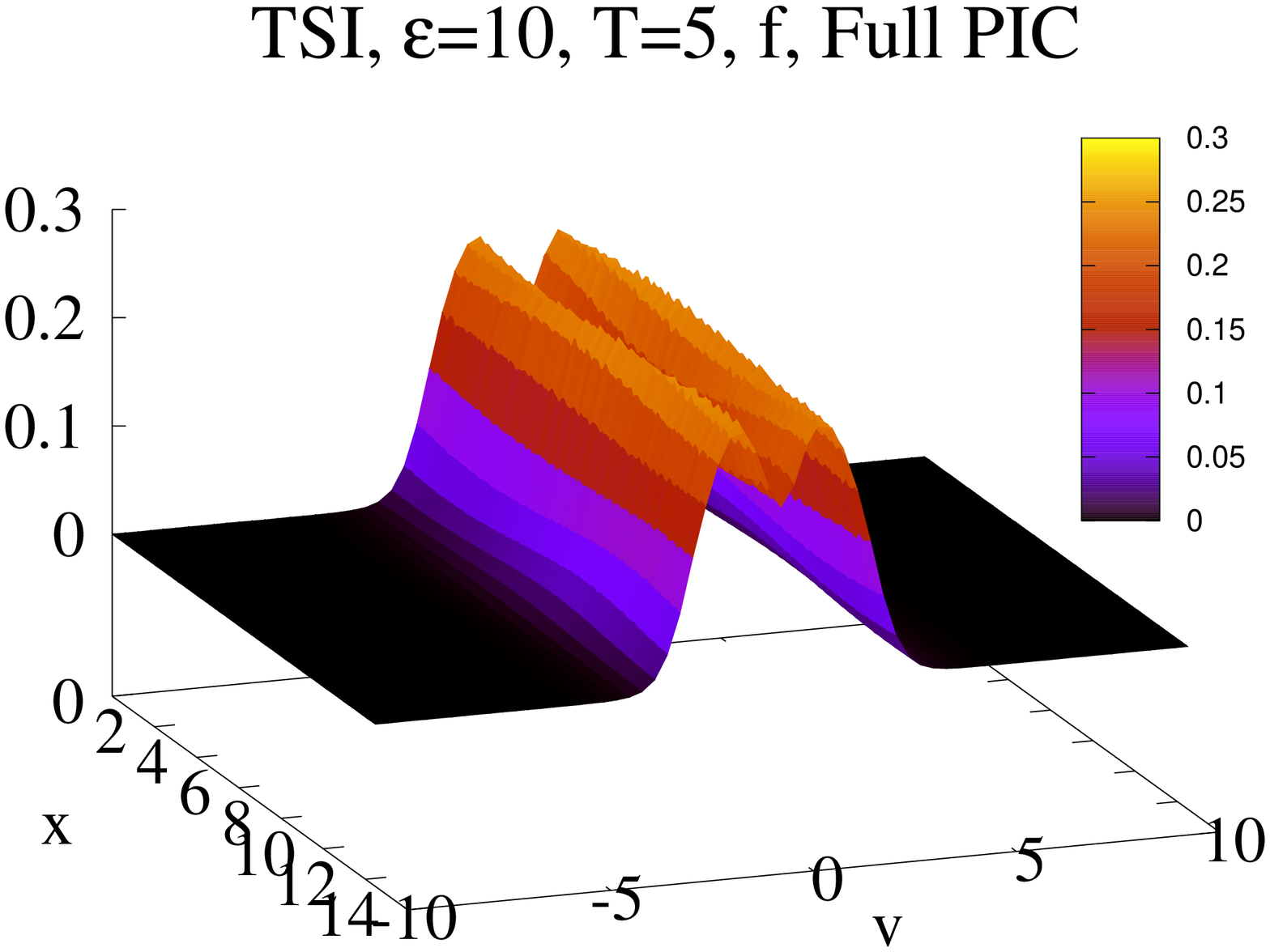}
\caption{Representation of $f(T=5,x,v)$, $\varepsilon=10$. Distribution function $f$ reconstructed from MiMa-Grid on the left, from MiMa-Part-2 on the middle and obtained by Full PIC on the right. $\Delta t=0.1$, $N_x=128$ and $N_p=10^6$ for both particle methods. $N_x=N_v=512$ and $\Delta t\approx 6\times 10^{-4}$ for MiMa-Grid.}
\label{fig3Deps10} 
\end{figure}

\begin{figure} [H]
\includegraphics[angle=-90,width=0.4\textwidth]{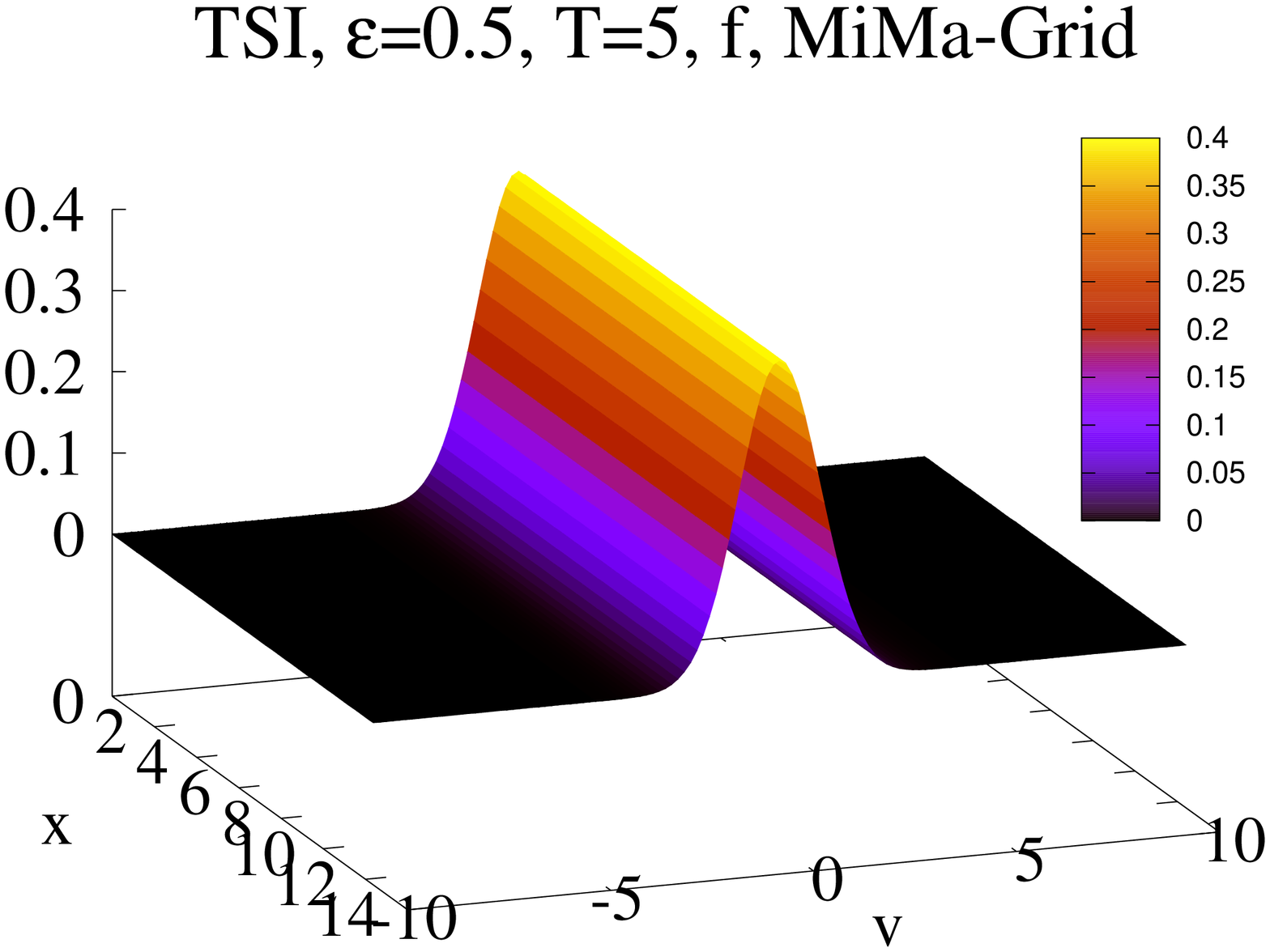}~~\hspace{-1.6cm}~~
\includegraphics[angle=-90,width=0.4\textwidth]{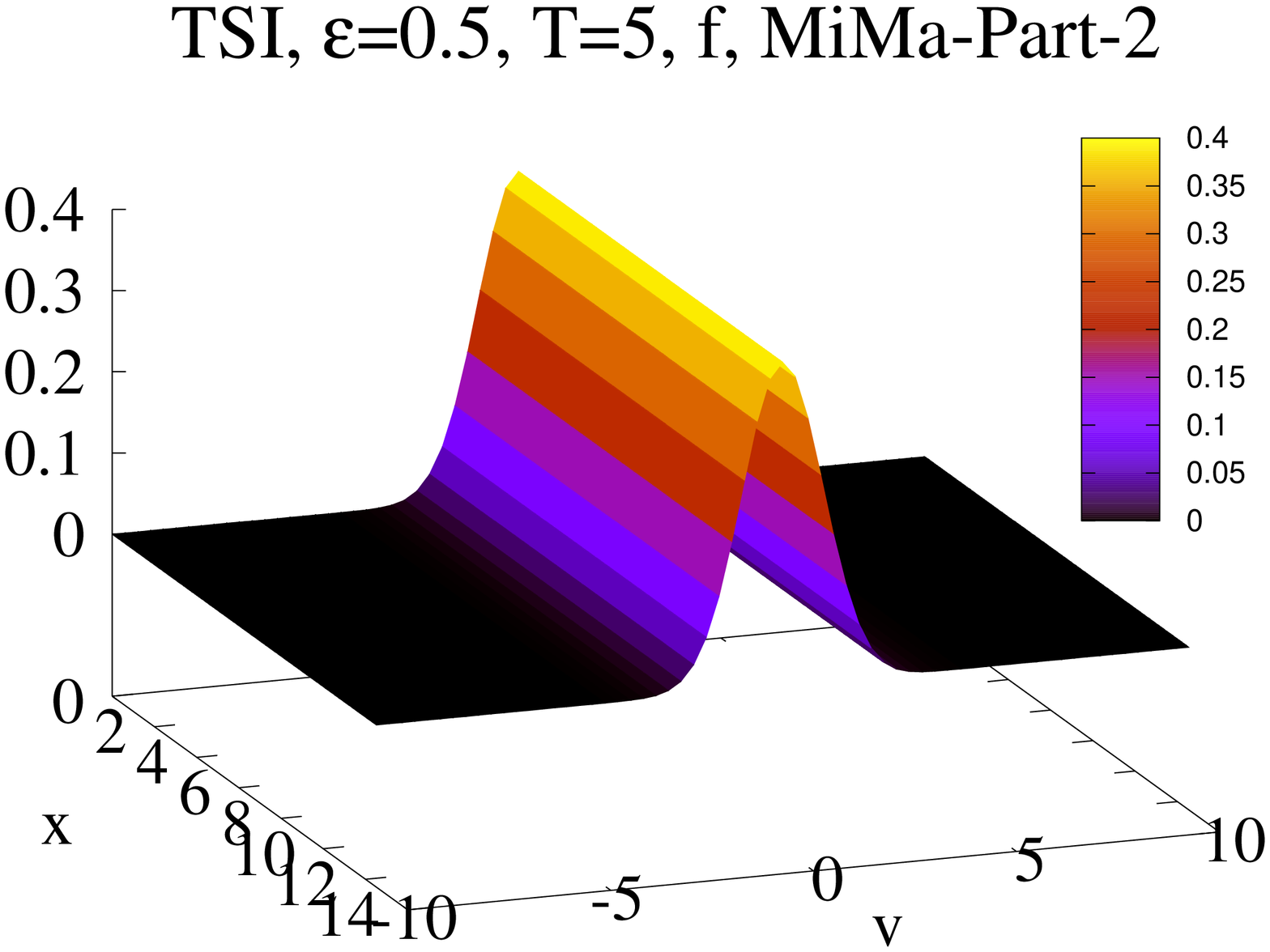}~~\hspace{-1.6cm}~~
\includegraphics[angle=-90,width=0.4\textwidth]{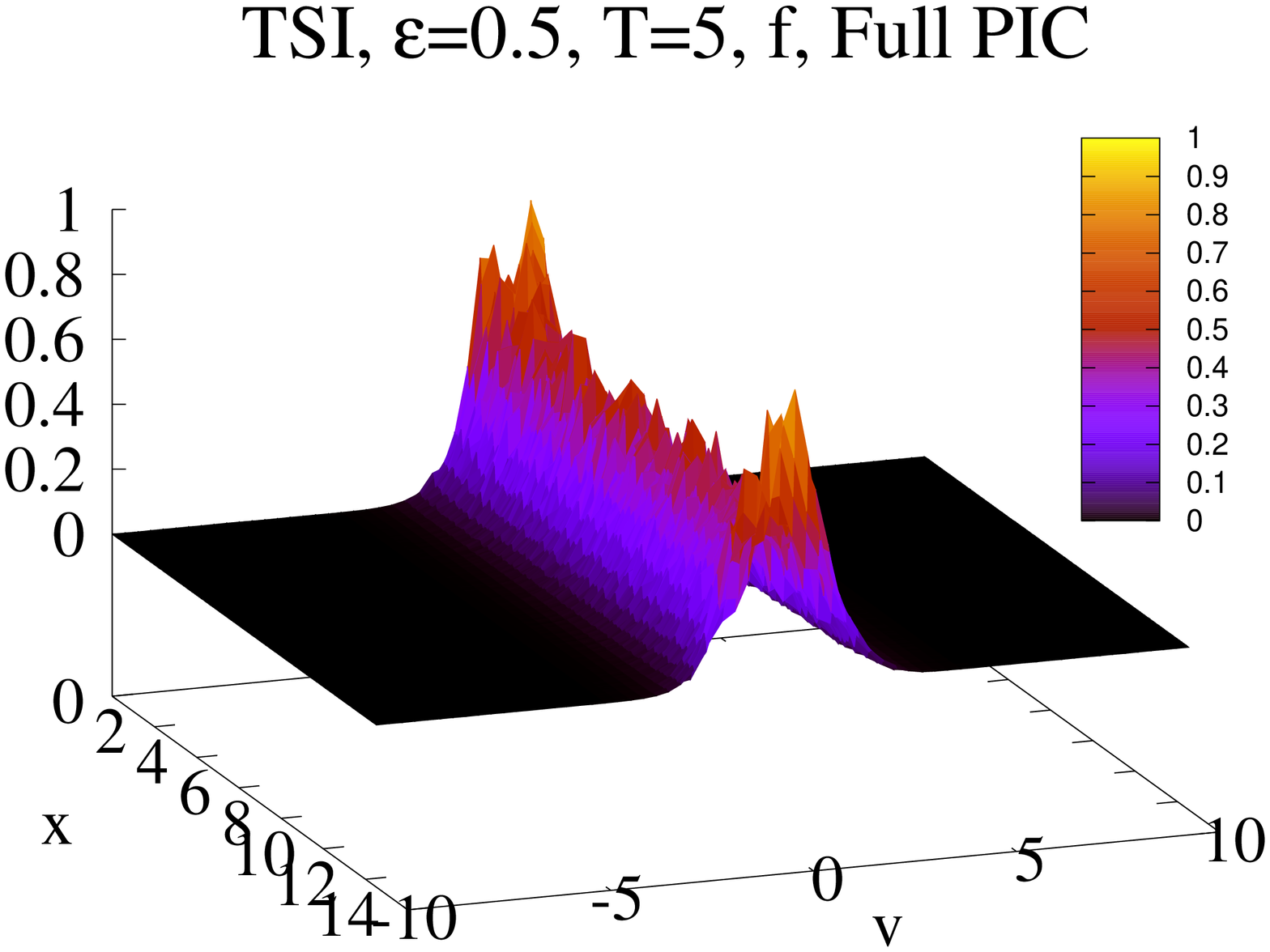}
\caption{Representation of $f(T=5,x,v)$, $\varepsilon=0.5$. Distribution function $f$ reconstructed from MiMa-Grid on the left, from MiMa-Part-2 on the middle and obtained by Full PIC on the right. $\Delta t=0.01$, $N_x=256$ and $N_p=10^5$ for both particle methods. $N_x=N_v=512$ and $\Delta t\approx 6\times 10^{-4}$ for MiMa-Grid.}
\label{fig3Deps05} 
\end{figure}

For smaller values of $\varepsilon$, we compare the four AP schemes (MiMa-Part-1, MiMa-Part-2, MiMa-Grid and Moment G.) to the limit scheme. Results for $\varepsilon=10^{-1}$ are given in Figure \ref{figTSIdif}. Parameters are the following: $\Delta t=10^{-3}$ , $N_x=128$ and $N_p=10^4$ for particle methods, $\Delta t=0.1\Delta x^2\approx 3.5\times 10^{-3}$ and $N_x=N_v=64$ for MiMa-Grid. As in the Landau damping case, MiMa-Part-2 gives the best result comparing to the reference MiMa-Grid. Finally, results for $\varepsilon=10^{-4}$ are given in Figure \ref{figTSIdiflim}, where we have $\Delta t=10^{-2}$, $N_x=128$ and $N_p=100$ for particle methods, $\Delta t=0.1\Delta x^2\approx 3.5\times 10^{-3}$ and $N_x=N_v=64$ for MiMa-Grid. The asymptotic regime is well recovered by all these AP methods.

\begin{minipage}[t]{0.48\textwidth}
\begin{figure} [H]
\includegraphics[angle=-90,width=\textwidth]{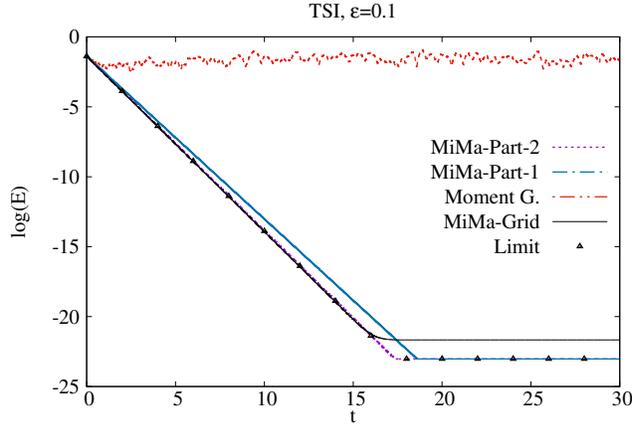}~~~~
\caption{Time history of the electric energy, $\varepsilon=0.1$. $\Delta t=10^{-3}$, $N_x=128$ and $N_p=10^4$ for the four particle methods.}% $\Delta t=0.1\Delta x^2$ and $N_x=N_v=64$ for MiMa-Grid.}
\label{figTSIdif} 
\end{figure}
\end{minipage}
~~
\begin{minipage}[t]{0.48\textwidth}
\begin{figure} [H]
\includegraphics[angle=-90,width=\textwidth]{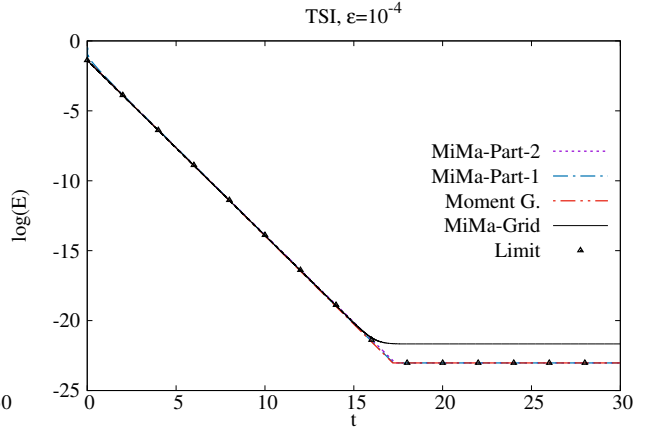}
\caption{Time history of the electric energy, $\varepsilon=10^{-4}$. $\Delta t=0.01$, $N_x=128$ and $N_p=100$ for the four particle methods.}% $\Delta t=0.1\Delta x^2$ and $N_x=N_v=64$ for MiMa-Grid.}
\label{figTSIdiflim} 
\end{figure}
\end{minipage}

\section{Conclusion}\label{sec:ccl}
\setcounter{equation}{0}

In this paper, we have presented new micro-macro models for the kinetic radiative transport equation (RTE), as well as for the Vlasov-Poisson-BGK system, in the diffusion scaling with periodic boundary conditions. First-order in time and second-order in time models are derived, and their Lagrangian discretizations are detailed. The obtained schemes are proved to degenerate into implicit discretizations of the limit model (the diffusion equation in the RTE case and the drift-diffusion equation in the Vlasov-Poisson-BGK case) when $\varepsilon\to 0$. This asymptotic property is shown in the numerical results too.

Moreover, thanks to the use of particle methods for the microscopic equation, the numerical cost is reduced when $\varepsilon$ diminishes. Finally, compared to a standard PIC method (where $f$ is represented by particles, and not $g$), the numerical noise is reduced. 

In future works, we would like to extend the Monte-Carlo approach proposed for the hydrodynamic limit of Vlasov-BGK in \cite{cdl} to the diffusion and the drift-diffusion limits.

\appendix

\section{Time discretization for Eulerian schemes }\label{subsec:eulerian}
\setcounter{equation}{0}

We present the time discretization of (\ref{eq:micromacro_initial}) having in spirit a Eulerian 
discretization of the phase space.  Obviously, the numerical scheme proposed 
in \cite{bennoune, cl, lm} works well. Now, (\ref{eq:mm_micro}) also provides a numerical scheme that we will exploit in this appendix.

Let us consider staggered grids in the phase-space domain and adopt the following notations: 
%$L_x\in\mathbb{R}$ (resp. $L_v\in\mathbb{R}$) is the domain length in space (resp. in velocity) and we consider two grids in space 
$\mbox{x}_i=i\Delta x$ and $\mbox{x}_{i+1/2}=i\Delta x+\Delta x/2$, $i\in\mathbb{N}$, define two grids in space and $\mbox{v}_j=j\Delta v$, $j\in\mathbb{N}$, defines a grid in velocity, where $\Delta x$ (resp. $\Delta v$) is the step in space (resp. in velocity). Time is also discretized with a time step $\Delta t$ and we note $t^n=n\Delta t$, $n\in\mathbb{N}$.
The density $\rho$ is discretized on the first space grid: $\rho_i^n$ approximates $\rho(t^n,\mbox{x}_i)$, whereas the perturbation $g$ is discretized on the second one: $g_{i+1/2,j}^n$ approximates $g(t^n,\mbox{x}_{i+1/2},\mbox{v}_j)$.

Let an approximation $D$ of the spatial derivative, the numerical scheme we propose consists in computing $g_{i+1/2,j}^{n+1}$ with
\begin{eqnarray}
g^{n+1}_{i+1/2, j} &=& e^{-\Delta t/\varepsilon^2} g^{n+1}_{i+1/2, j} - \varepsilon(1-e^{-\Delta t/\varepsilon^2}) \left[  v_j \frac{\rho^n_{i+1}-\rho_i^n}{\Delta x}  \right. \nonumber\\
\label{s2}
&& \left. + (I-\Pi)\left(v^+_j (D^-_x g^n)_{i+1/2, j}+v^-_j (D^+_x g^n)_{i+1/2, j}   \right)   \right], \nonumber\\
\end{eqnarray}
where $(\Pi h)_{i+1/2, j} = (\sum_j h_{i+1/2, j}\Delta v)$,
and then to compute $\rho_i^{n+1}$ with
\be
\label{macrod}
\rho_i^{n+1} = \rho_i^n -\frac{\Delta t}{\varepsilon} \sum_j \left(v_j  \frac{g^{n+1}_{i+1/2, j}-g^{n+1}_{i-1/2, j}}{\Delta x} \right)\Delta v.
\ee

\begin{prop}
The scheme given by (\ref{s2})-(\ref{macrod}) enjoys the AP property, \textit{i.e.} it satisfies the following properties 
%is a first-order AP scheme 
%for system (\ref{eq:mm_macro})-(\ref{eq:mm_micro})
\begin{itemize}
\item for fixed $\varepsilon>0$, the scheme is a first-order (in time) approximation of the original model (\ref{eq:etrbgk}), % (\ref{eq:mm_macro})-(\ref{eq:mm_micro}), 
\item for fixed $\Delta t>0$, the scheme degenerates into an explicit first-order (in time) scheme of \eqref{eq:diff}. 
\end{itemize}
%discretization of the diffusion equation $\partial_t\rho-\partial_{xx}\rho=0$ when $\varepsilon\to 0$. A necessary stability condition is $\Delta t={\cal O}(\Delta x^2)$, coming from the diffusion term.
\label{prop:euler_1st_exp}
\end{prop}

\begin{proof}
We observe easily that when $\varepsilon$ goes to zero, (\ref{s2}) gives
$$
g^{n+1}_{i+1/2, j} = -\varepsilon  v_j \frac{\rho^n_{i+1}-\rho_i^n}{\Delta x}  +{\cal O}(\varepsilon^2), 
$$ 
which, injected in the time discretization (\ref{macrod}) for $\rho$, 
%\be
%\label{macrod0}
%\rho_i^{n+1} = \rho_i^n -\frac{\Delta t}{\varepsilon}\partial_x\langle vg^{n+1}_{i+1/2}\rangle , 
%\ee
%or in the fully discretized expression
%\be
%\label{macrod}
%\rho_i^{n+1} = \rho_i^n -\frac{\Delta t}{\varepsilon} \sum_j \left(v_j  \frac{g^{n+1}_{i+1/2, j}-g^{n+1}_{i-1/2, j}}{\Delta x} \right)\Delta v, 
%\ee
gives up to terms of order ${\cal O}(\varepsilon^2)$
$$
\rho_i^{n+1} = \rho_i^n + \Delta t \left(\sum_j v^2_j \Delta v\right) \frac{\rho^n_{i+1}-2\rho_i^n+\rho^n_{i-1}}{\Delta x^2} . 
$$
Since $\sum_j v^2_j \Delta v$ is an approximation of $\int_{-1}^1 v^2 \dd v=1/3$, we obtain a consistent 
discretization of the diffusion equation. 
\end{proof}

Proposition \ref{prop:euler_1st_exp} is of big interest for impliciting the diffusion term $\partial_{xx}\rho$. Indeed, let us rewrite \eqref{s2} as follows 
$$
g_{i+1/2,j}^{n+1} = - \varepsilon(1-e^{-\Delta t/\varepsilon^2})   v_j \frac{\rho^n_{i+1}-\rho_i^n}{\Delta x}  + h_{i+1/2,j}, 
$$
with $h_{i+1/2,j}=e^{-\Delta t/\varepsilon^2} g^{n+1}_{i+1/2, j} - \varepsilon(1-e^{-\Delta t/\varepsilon^2}) \left[ (I-\Pi)\left(v^+_j (D^-_x g^n)_{i+1/2, j}+v^-_j (D^+_x g^n)_{i+1/2, j}   \right)   \right].$
Injecting this relation into the macro part, we get 
\begin{eqnarray}
\rho^{n+1}_i &=& \rho_i^n + \Delta t (1-e^{-\Delta t/\varepsilon^2})\sum_j ( v_j^2)\Delta v  \frac{\rho^n_{i+1}-2\rho_i^n + \rho^n_{i-1}}{\Delta x^2} - \frac{\Delta t}{\varepsilon} \sum_j \left( v_j \frac{h_{i+1/2, j}-h_{i-1/2, j}}{\Delta x}\right) \Delta v. ~~\nonumber\\
\label{eq:euler_1st_manoeuvre}
\end{eqnarray}
Since $h_{i+1/2, j}= {\cal O}(\varepsilon^2)$ as $\varepsilon$ goes to zero after two iterations, the asymptotic preserving property is ensured. Moreover, the diffusion term can now be chosen as implicit, so that the macro equation becomes
\begin{eqnarray}
\rho^{n+1}_i &=& \rho_i^n + \Delta t (1-e^{-\Delta t/\varepsilon^2}) \sum_j ( v_j^2)\Delta v \frac{\rho^{n+1}_{i+1}-2\rho_i^{n+1} + \rho^{n+1}_{i-1}}{\Delta x^2}- \frac{\Delta t}{\varepsilon} \sum_j \left( v_j \frac{h_{i+1/2, j}-h_{i-1/2, j}}{\Delta x}\right) \Delta v,~~\nonumber\\
\label{macrod_cn}
\end{eqnarray}
and the scheme is now free from the usual diffusion condition on the time step.

The algorithm finally writes 

\begin{algo}~~ 
\begin{itemize}
\item Initialize $g_{i+1/2,j}^0$ and $\rho_i^0$.

At each time step:
\item Advance micro part with (\ref{s2}).
\item Advance macro part with (\ref{macrod_cn}).
\end{itemize}
\label{algo:euler_1st}
\end{algo}  

And we have the following result.
\begin{prop}
The scheme given by (\ref{s2})-(\ref{macrod_cn}) enjoys the AP property, \textit{i.e.} it satisfies the following properties 
%is a first-order AP scheme 
%for system (\ref{eq:mm_macro})-(\ref{eq:mm_micro})
\begin{itemize}
\item for fixed $\varepsilon>0$, the scheme is a first-order (in time) approximation of the original model (\ref{eq:etrbgk}), %(\ref{eq:mm_macro})-(\ref{eq:mm_micro}), 
\item for fixed $\Delta t>0$, the scheme degenerates into an implicit first-order (in time) scheme of \eqref{eq:diff}. 
\end{itemize}
%discretization of the diffusion equation $\partial_t\rho-\partial_{xx}\rho=0$ when $\varepsilon\to 0$. A necessary stability condition is $\Delta t={\cal O}(\Delta x^2)$, coming from the diffusion term.
\label{prop:euler_1st_imp}
\end{prop}

We do not present here the extension to the Vlasov-Poisson-BGK case, but it is straightforward.

\section{Moment guided}\label{app:mgm}
\setcounter{equation}{0}

In this section, we present the adaptation of the moment guided particle method proposed in \cite{ddp} to our context. For the sake of simplicity, we present it in the RTE case but note that these computations can also be extended to the Vlasov-Poisson-BGK case, without difficulty.

%Starting from (\ref{macrod0}), one can compute $\rho^{n+1}_i$ from known quantities at time $t^n$, using $g=f-\rho$. 
The kinetic equation on $f$ has to be reformulated to avoid the singularity linked to the transport term. 
To do that, we proceed as previously, but from (\ref{eq:etrbgk}).
%$$
%\partial_t f + \frac{1}{\varepsilon}v\partial_x f  =\frac{1}{\varepsilon^2}(\rho M-f). 
%$$
Indeed, we rewrite equation (\ref{eq:etrbgk}) as 
$$
\partial_t (e^{t/\varepsilon^2}f) =   \frac{e^{t/\varepsilon^2}}{\varepsilon}\left[-v\partial_x f  + \frac{1}{\varepsilon}\rho\right], 
$$
and we integrate between $t^n$ and $t^{n+1}$ to get 
$$
f(t^{n+1}) = e^{-\Delta t/\varepsilon^2}f(t^n) - \frac{e^{-t^{n+1}/\varepsilon^2}}{\varepsilon} \int_{t^n}^{t^{n+1}} e^{t/\varepsilon^2} \left[v\partial_x f - \frac{1}{\varepsilon}\rho\right] \dd t.
$$
We make the following approximation 
$$
f^{n+1} = e^{-\Delta t/\varepsilon^2}f^n - \varepsilon (1-e^{-\Delta t/\varepsilon^2}) \left[v\partial_x f^n - \frac{1}{\varepsilon}\rho^{n}\right],
$$
where $f^n\approx f(t^n)$ and $\rho^{n}\approx \rho(t^{n})$, $\forall n$.

Making appear the discrete time derivative enables to write 
$$
\frac{f^{n+1} -f^n}{\Delta t}= \frac{e^{-\Delta t/\varepsilon^2}-1}{\Delta t}f^n - \varepsilon \frac{1-e^{-\Delta t/\varepsilon^2}}{\Delta t} \left[v\partial_x f^n - \frac{1}{\varepsilon}\rho^{n}\right], 
$$
which we approximate by 
\begin{equation}
\partial_t f= \frac{e^{-\Delta t/\varepsilon^2}-1}{\Delta t}f - \varepsilon \frac{1-e^{-\Delta t/\varepsilon^2}}{\Delta t} \left[v\partial_x f - \frac{1}{\varepsilon}\rho\right]. 
\label{eq:appB}
\end{equation}

Following the spirit of the moment guided method (see \cite{ddp}), this equation is coupled with the macro one, that is
\begin{eqnarray*}
\partial_t \rho+\frac{1}{\varepsilon}\partial_x\langle vf\rangle&=&0,\\
\partial_t f+\varepsilon \frac{1-e^{-\Delta t/\varepsilon^2}}{\Delta t}v\partial_x f&=& \frac{e^{-\Delta t/\varepsilon^2}-1}{\Delta t}f + \frac{1-e^{-\Delta t/\varepsilon^2}}{\Delta t} \rho.
\end{eqnarray*}

To derive an AP scheme for this latter system satisfied by $(\rho,f)$, we adapt the strategy presented in \cite{ddp} to our diffusion framework. To do so, we first remark that $\langle vf\rangle=\langle vg\rangle$ and using the expression of $g$ obtained by (\ref{eq:mm_micro}), we get the following approximation for the macro flux (considered implicit in time)
$$
\frac{1}{\varepsilon}\partial_x\langle vg^{n+1}\rangle =  - (1-e^{-\Delta t/\varepsilon^2})\partial_{xx}\rho^n+\frac{1}{\varepsilon}e^{-\Delta t/\varepsilon^2}\partial_x\langle vg^n\rangle.
$$
Then, we get the following scheme for $\rho$
\begin{equation}
\rho^{n+1}=\rho^n+\Delta t(1-e^{-\Delta t/\varepsilon^2})\partial_{xx}\rho^n-\frac{\Delta t}{\varepsilon}e^{-\Delta t/\varepsilon^2}\partial_x\langle vf^n\rangle.
\label{eq:rhomg}
\end{equation}

A Lagrangian method can be used to approximate the equation on $f$. As for the micro-macro scheme, we use a splitting procedure
\begin{itemize}
\item solve $\partial_t f + \varepsilon \frac{1-e^{-\Delta t/\varepsilon^2}}{\Delta t}  v\partial_x f = 0$ 
\item solve $\partial_t f =  \frac{e^{-\Delta t/\varepsilon^2} -1}{\Delta t}f+ \frac{1-e^{-\Delta t/\varepsilon^2}}{\Delta t} \rho$.    
\end{itemize}
To do that, the transport part is solved with the (non stiff) characteristics 
\be
\label{carxf}
\dot{x}_k(t) =  \varepsilon \frac{1-e^{-\Delta t/\varepsilon^2}}{\Delta t}  v_k(t).
\ee
The source part is solved using the equation satisfied by the weights 
\be
\label{weightf}
\dot{\omega}_k(t) = \frac{e^{-\Delta t/\varepsilon^2} -1}{\Delta t}\omega_k(t) + \frac{1-e^{-\Delta t/\varepsilon^2}}{\Delta t} \rho(t, x_k(t)).
\ee
%Equation \eqref{carxf} can be solved using an explicit Euler scheme for example. 
%Equation \eqref{weightf} has to be solved 
%using an implicit scheme on the last term $\rho$. %This can be done easily since $\rho^{n+1}$ has been computed at the beginning of the time step so that 
%This can be done by computing $\rho^{n+1}$ at the beginning of the time step. For that, let us integrate in $v$ the reformulated equation (\ref{eq:appB}) to obtain
%$$
%\partial_t\rho+\varepsilon\frac{1-e^{-\Delta t/\varepsilon^2}}{\Delta t}\partial_x\langle vf\rangle=0,
%$$
%which does not contain any stiff terms. Then, we apply an explicit Euler scheme, knowing $\rho^n$ and $f^n$:
%$$
%\rho^{n+1}=\rho^n-\varepsilon(1-e^{-\Delta t/\varepsilon^2})\partial_x\langle vf^n\rangle.
%$$
%
%Now, the scheme of the weights writes 
%$$
%\omega_k^{n+1} = e^{-\Delta t/\varepsilon^2} \omega_k^n + (1-e^{-\Delta t/\varepsilon^2}) \rho^{n+1}(x_k^{n+1}), 
%$$
%which can be viewed as an exact integration of $\partial_t f = \frac{1}{\varepsilon^2} (\rho - f)$. 

The last step consists in matching the moment of $f^{n+1}$ obtained by the particle method with $\rho^{n+1}$ obtained with (\ref{eq:rhomg}). This can be done using the techniques 
proposed in \cite{ccl}. Indeed, considering the function $g=f-\rho$, its weight can be written as 
$$
\gamma_k = \omega_k - \beta_k, \;\; \mbox{ with } \;\;  \beta_k=\rho(x_k)\frac{L_xL_v}{N_{p}}.
$$  
Then, we apply the discrete version of $(I-\Pi)$ to the weights $\gamma_k$ as in \cite{ccl} 
$$
\omega_k^{new} = \beta_k +(I-\Pi)( \omega_k - \beta_k). 
$$

\paragraph{Acknowledgements}~~\\

N. Crouseilles and M. Lemou are supported by the French ANR project MOONRISE ANR-14-CE23-0007-01 and by the Enabling Research EUROFusion project CfP-WP14-ER-01/IPP-03. A. Crestetto is supported by the French ANR project ACHYLLES  ANR-14-CE25-0001.

%################################################################
%################################################################
%#####################BIBLIOGRAPHY###############################
%################################################################
%################################################################


\begin{thebibliography}{1}
%

\bibitem{bat}
{\sc N. Ben Abdallah, M.-L. Tayeb}, 
\textit{Diffusion Approximation for the one dimensional Boltzmann-Poisson system}, 
Discrete and Continuous of Dynamical Systems-Series B {\bf 4}, pp. 1129-1142 (2004). 


\bibitem{bennoune}
{\sc M. Bennoune, M. Lemou, L. Mieussens}, 
\textit{Uniformly stable numerical schemes for the Boltzmann equation preserving 
the compressible Navier-Stokes asymptotics}, 
J. Comput. Phys. {\bf 227}, pp. 3781-3803 (2008). 




%
%\bibitem{bt}
%{\sc C. Berthon, R. Turpault}, 
%\textit{A numerical correction of the {M1}-model in the diffusive limit}, 
%Discrete and Continuous Dynamical Systems Series S {\bf 5}, pp. 245-255 (2012). 

%
\bibitem{birdsall}
{\sc C.K. Birdsall, A.B. Langdon}, 
\textit{Plasma Physics via Computer Simulation}, 
CRC Press (2004). 


%
\bibitem{buet}
{\sc C. Buet, S. Cordier}, 
\textit{Asymptotic preserving scheme and numerical methods for radiative hydrodynamic models}, 
Comptes Rendus Math\'ematique {\bf 338}, pp. 951-956 (2004). 


%\url{http://christophe.buet.pagesperso-orange.fr/a01.pdf}
%
\bibitem{ccl}
{\sc A. Crestetto, N. Crouseilles, M. Lemou}, 
\textit{Kinetic/Fluid micro-macro numerical schemes for Vlasov-Poisson-BGK equations using particles}, 
Kin. Rel. Models {\bf 5}, pp. 787-816 (2012). 
%
\bibitem{ccl2}
{\sc A. Crestetto, N. Crouseilles, M. Lemou}, 
\textit{Asymptotic-Preserving scheme based on a Finite Volume/Particle-In-Cell coupling for Boltzmann-BGK-like equations in the diffusion scaling}, 
Finite Volumes for Complex Applications VII - Elliptic, Parabolic and Hyperbolic Problems, Springer Proceedings in Mathematics and Statistics {\bf 78}, pp. 827-835 (2014). 
%

\bibitem{cdl}
{\sc N. Crouseilles, G. Dimarco, M. Lemou}, 
\textit{ Asymptotically stable and time diminishing schemes for rarefied gas dynamics}, 
accepted in Kin. Rel. Models. 


\bibitem{cl}
{\sc N. Crouseilles, M. Lemou}, 
\textit{An asymptotic preserving scheme based on a micro-macro 
decomposition for collisional Vlasov equations: diffusion and high-field scaling limits}, 
Kin. Rel. Models {\bf 4}, pp. 441-477 (2011). 

%
\bibitem{dd}
{\sc P. Degond, G. Dimarco}, 
\textit{Fluid simulations with localized Boltzmann upscaling by Direct Simulation Monte-Carlo}, 
J. Comput. Phys.  {\bf 231}, pp. 2414-2437 (2012).


%
\bibitem{ddp}
{\sc P. Degond, G. Dimarco, L. Pareschi}, 
\textit{The Moment Guided Monte Carlo Method}, 
International Journal for Numerical Methods in Fluids {\bf 67}, pp. 189-213 (2011).

\bibitem{dgp}
{\sc P. Degond, T. Goudon, F. Poupaud}, 
\textit{Diffusion limit for non homogeneous and non-micro-
reversible processes}, 
Indiana Univ. Math. J., {\bf 49} pp. 1175-1198 (2000).

%
\bibitem{dp}
{\sc G. Dimarco, L. Pareschi, V. Rispoli}, 
\textit{Implicit-Explicit Runge-Kutta schemes for the Boltzmann-Poisson system for semiconductors}, 
Communications in Computational Physics {\bf 15}, pp. 1291-1319 (2014).


%
\bibitem{gjl}
{\sc F. Golse, S. Jin, C. D. Levermore}, 
\textit{A domain decomposition analysis for a two-scale linear transport problem}, 
Mathematical Modelling and Numerical Analysis {\bf 37}, pp. 869-892 (2003).

%
\bibitem{jin}
{\sc S. Jin}, 
\textit{Efficient Asymptotic-Preserving (AP) schemes for some multiscale kinetic equations}, 
SIAM J. Sci. Comput. {\bf 21}, pp. 441-454 (1999).

\bibitem{jpt}
{\sc S. Jin, L. Pareschi, G. Toscani}, 
\textit{Uniformly accurate diffusive relaxation schemes for multiscale transport equations}, 
SIAM J. Numerical Analysis {\bf 38}, pp. 913-936, (2000).

%
\bibitem{klar}
{\sc A. Klar}, 
\textit{An Asymptotic-Induced Scheme For Nonstationary Transport Equations In The Diffusive Limit}, 
SIAM Journal of Numerical Analysis {\bf 35}, pp. 1073-1094 (1998).

\bibitem{bt}
{\sc K. Krycki, C. Berthon, M. Frank, R. Turpault}, 
\textit{Asymptotic preserving numerical schemes for a nonclassical radiation transport model for atmospheric clouds}, 
Math. Meth. Applied Scie. {\bf 36}, pp. 2101-2116 (2013). 


\bibitem{larsen}
{\sc E. W. Larsen, J. B. Keller}, 
\textit{Asymptotic solution of neutron transport problems for small
mean free paths}, 
J. Math. Phys, {\bf 15} pp. 75-81 (1974).
%
\bibitem{lemou-note}
{\sc M. Lemou}, 
\textit{Relaxed micro-macro schemes for kinetic equations}, 
Comptes Rendus Math\'ematique {\bf 348}, pp. 455-460 (2010). 

%
%\bibitem{mlb}
%{\sc M. Lemou, F. M\'ehats}, 
%\textit{A boundary matching micro/macro decomposition for kinetic equations}, 
%Comptes Rendus Math\'ematique {\bf 349}, pp. 479-484, (2011). 
\bibitem{mlb}
{\sc M. Lemou, F. M\'ehats}, 
\textit{Micro-Macro Schemes for Kinetic Equations Including Boundary Layers}, 
SIAM J. Sci. Comput. {\bf 34}, pp. 734-760 (2012). 


%
\bibitem{lm}
{\sc M. Lemou, L. Mieussens}, 
\textit{A new asymptotic preserving scheme based on micro-macro formulation 
for linear kinetic equations in the diffusion limit}, 
SIAM J. Sci. Comp. {\bf 31}, pp. 334-368 (2008). 

%
\bibitem{tallec}
{\sc P. Le Tallec, F. Mallinger}, 
\textit{Coupling Boltzmann and Navier-Stokes by half-fluxes}, 
J. Comput. Phys.  {\bf 136}, pp. 51-67 (1997). 

\bibitem{liu}
{\sc T.-P. Liu, S.-H. Yu}, 
\textit{Boltzmann Equation: Micro-Macro Decompositions and Positivity of Shock Profiles}, 
Comm. Math. Phys. {\bf 246}, pp. 133-179 (2004). 

\bibitem{np}
{\sc G. Naldi, L. Pareschi}, 
\textit{Numerical schemes for kinetic equations in diffusive regimes}, 
Applied Math. Letters {\bf 11}, pp. 29-35, (1998). 


%
\bibitem{tiwari}
{\sc S. Tiwari, A. Klar, S. Hardt}, 
\textit{A particle-particle hybrid method for kinetic and continuous equations}, 
J. Comput. Phys.  {\bf 228}, pp. 7109-7124 (2009). 
\end{thebibliography}
\end{document}